\documentclass[a4paper,
fontsize=11pt,%
oneside,%
numbers=enddot]{scrartcl}
\KOMAoptions{DIV=12}    
\usepackage[utf8]{inputenc}
\usepackage[T1]{fontenc}
\usepackage{textcomp}
\usepackage[fleqn]{amsmath}
\usepackage{amsthm}
\usepackage{amssymb}
\usepackage{amsfonts}
\usepackage{libertine}
\usepackage[libertine,
slantedGreek,
nosymbolsc,
nonewtxmathopt,
subscriptcorrection]{newtxmath}
\usepackage[scaled=0.95,varqu,varl]{inconsolata}
\usepackage[scr=boondox]{mathalpha}   
\usepackage{euscript}   
\usepackage{leftindex}
\usepackage{subcaption}
\allowdisplaybreaks[1]  
\numberwithin{equation}{section}    
\usepackage[svgnames,hyperref]{xcolor}
\definecolor{dblue}{HTML}{0455BF}
\definecolor{dgreen}{HTML}{02724A}
\definecolor{dgreen2}{HTML}{025951}
\definecolor{dred}{HTML}{D90404}
\definecolor{dviolet}{HTML}{42208C}
\definecolor{labelkey}{HTML}{025951}
\definecolor{refkey}{HTML}{025951}
\addtokomafont{section}{\centering}

\usepackage{enumitem}
\setlist{itemsep=-2.0pt}

\usepackage{upref}  
\usepackage{hyperref}
\hypersetup{colorlinks=true,
linktocpage=true,
linkcolor=dblue,
citecolor=dgreen,
urlcolor=dred,
pdfencoding=auto,
hypertexnames=false}
\makeatletter
\g@addto@macro\th@plain{
\thm@headfont{\bfseries\sffamily}
\thm@notefont{}}
\g@addto@macro\th@definition{
\thm@headfont{\bfseries\sffamily}
\thm@notefont{}}
\g@addto@macro\th@remark{
\thm@headfont{\bfseries\sffamily}
\thm@notefont{}}
\makeatother
\theoremstyle{plain}
\newtheorem{theorem}{Theorem}[section]
\newtheorem{proposition}[theorem]{Proposition}

\newtheorem{lemma}[theorem]{Lemma}

\theoremstyle{definition}
\newtheorem{definition}[theorem]{Definition}
\newtheorem{example}[theorem]{Example}

\theoremstyle{remark}
\newtheorem{remark}[theorem]{Remark}
\newtheorem{notation}[theorem]{Notation}

\usepackage[extdef=true]{delimset}
\DeclareMathDelimiterSet{\scal}[2]{
\selectdelim[l]<{#1}
\mathpunct{}\selectdelim[p]|
{#2}\selectdelim[r]>}
\renewcommand{\pair}{\delimpair<{[.],}>}
\newcommand{\menge}[2]{\bigl\{{#1}\mid{#2}\bigr\}} 
\DeclareMathDelimiterSet{\Menge}[2]{\selectdelim[l]\{
{#1}\selectdelim[m]|{#2}\selectdelim[r]\}}

\makeatletter
\def\upintkern@{\mkern-7mu\mathchoice{\mkern-3.5mu}{}{}{}}
\def\upintdots@{\mathchoice{\mkern-4mu\@cdots\mkern-4mu}%
{{\cdotp}\mkern1.5mu{\cdotp}\mkern1.5mu{\cdotp}}%
{{\cdotp}\mkern1mu{\cdotp}\mkern1mu{\cdotp}}%
{{\cdotp}\mkern1mu{\cdotp}\mkern1mu{\cdotp}}}
\makeatother
\DeclareFontFamily{OMX}{mdbch}{}
\DeclareFontShape{OMX}{mdbch}{m}{n}{ <->s * [0.8]  mdbchr7v }{}
\DeclareFontShape{OMX}{mdbch}{b}{n}{ <->s * [0.8]  mdbchb7v }{}
\DeclareFontShape{OMX}{mdbch}{bx}{n}{<->ssub * mdbch/b/n}{}
\DeclareSymbolFont{uplargesymbols}{OMX}{mdbch}{m}{n}
\SetSymbolFont{uplargesymbols}{bold}{OMX}{mdbch}{b}{n}
\DeclareMathSymbol{\upintop}{\mathop}{uplargesymbols}{82}
\DeclareMathSymbol{\upointop}{\mathop}{uplargesymbols}{"48}
\makeatletter
\renewcommand{\int}{\DOTSI\upintop\ilimits@}
\renewcommand{\oint}{\DOTSI\upointop\ilimits@}
\makeatother

\newcommand{\RR}{\mathbb{R}}

\newcommand{\DD}{\mathcal{D}}
\newcommand{\XX}{\mathcal{X}}
\newcommand{\YY}{\mathcal{Y}}
\newcommand{\EEE}{\ensuremath{\boldsymbol{\mathcal{E}}}}

\newcommand{\DDD}{\ensuremath{\boldsymbol{\mathcal{D}}}}

\newcommand{\pinf}{{+}\infty}
\newcommand{\minf}{{-}\infty}
\newcommand{\zeroun}{\intv[o]{0}{1}}

\newcommand{\RXX}{\intv{\minf}{\pinf}}
\newcommand{\RX}{\intv[l]0{\minf}{\pinf}}
\newcommand{\RP}{\intv[r]0{0}{\pinf}}
\newcommand{\RPP}{\intv[o]0{0}{\pinf}}
\newcommand{\RPX}{\intv0{0}{\pinf}}

\newcommand{\RMM}{\intv[o]0{\minf}{0}}

\newcommand{\emp}{\varnothing}

\newcommand{\minimize}[2]{\underset{\substack{{#1}}}
{\operatorname{minimize}}\;\;#2}

\newcommand{\pushfwd}%
{\ensuremath{\mbox{\Large$\,\triangleright\,$}}}

\DeclareMathOperator{\Argmin}{Argmin}

\newcommand{\Id}{\mathrm{Id}}

\DeclareMathOperator{\ran}{ran}

\DeclareMathOperator{\dom}{dom}
\DeclareMathOperator{\intdom}{int\,dom}

\DeclareMathOperator{\gra}{gra}
\DeclareMathOperator{\zer}{zer}
\DeclareMathOperator{\inte}{int}

\DeclareMathOperator{\Prox}{Prox}
\DeclareMathOperator{\prox}{prox}
\DeclareMathOperator{\proj}{proj}

\renewcommand{\leq}{\leqslant}
\renewcommand{\geq}{\geqslant}
\newcommand{\exi}{\exists\,}

\renewenvironment{abstract}{%
\vspace*{-0.50cm}
\small
\quotation%
\noindent%
{\normalfont\bfseries\sffamily
\nobreak\abstractname\ }%
}{%
\endquotation%
\medskip
}
\renewcommand{\abstractname}{Abstract.}
\newcommand\keywordsname{Keywords.}
\newenvironment{keywords}
{\renewcommand\abstractname{\keywordsname}\begin{abstract}}
{\end{abstract}}
\usepackage[auth-sc]{authblk}
\newcommand{\email}[1]{\href{mailto:#1}{\nolinkurl{#1}}}
\renewcommand*\Affilfont{\normalfont\normalsize}
\newcommand\affilcr{\protect\\ \protect\Affilfont}
\makeatletter
\renewcommand\AB@affilsepx{\protect\\[0.5em]}
\makeatother

\author[1]{Patrick L. Combettes}
\affil[1]{North Carolina State University
\affilcr
Department of Mathematics
\affilcr
Raleigh, NC 27695, USA
\affilcr
\email{plc@math.ncsu.edu}
}
\author[2]{Julien N. Mayrand}
\affil[2]{North Carolina State University
\affilcr
Department of Mathematics
\affilcr
Raleigh, NC 27695, USA
\affilcr
\email{jnmayran@ncsu.edu}
}

\begin{document}

\title{%
Lower Bounds on the Haraux Function\thanks{%
Contact author: P. L. Combettes.
Email: \email{plc@math.ncsu.edu}.
Phone: +1 919 515 2671.
This work was supported by the National Science Foundation 
under grant DMS-2513409. 
}}

\date{~}

\maketitle

\vspace{12mm}

\begin{abstract}
The Haraux function is an important tool in monotone
operator theory and its applications. One of its salient 
properties for a maximally monotone operator is to be valued in
$[0,+\infty]$ and to vanish only on the graph of the operator.
Sharper lower bounds for this function have been proposed in
specific cases. We derive lower bounds in the general
context of set-valued operators in reflexive real Banach spaces. 
These bounds are new, even for maximally monotone operators acting 
on Euclidean spaces, a scenario in which we show that they can be
better than existing ones. As a by-product, we obtain lower bounds
on the Fenchel--Young function in variational analysis. Several
examples are given and applications to composite monotone
inclusions are discussed.
\end{abstract}


\newpage

\section{Introduction}
\label{sec:1}

Throughout, $\XX\neq\{0\}$ is a reflexive real Banach space with
topological dual $\XX^*$ and $\pair{\cdot}{\cdot}$ denotes the
canonical duality pairing. The power set of $\XX^*$ is denoted by
$2^{\XX^*}$. See Section~\ref{sec:2} for further notation.

Let $A\colon\XX\to 2^{\XX^*}$ be an operator and let
$\gra A=\menge{(x,x^*)\in\XX\times\XX^*}{x^*\in Ax}$ be its graph. 
The \emph{Haraux function} associated with $A$ is
\begin{equation} 
\label{e:1}
H_A\colon\XX\times\XX^*\to\RXX\colon(x,u^*)\mapsto
\sup_{(y,y^*)\in\gra A}\pair{x-y}{y^*-u^*}.
\end{equation} 
This function was originally conceived by Haraux in 1974 in
unpublished notes as a tool to generalize the notions of 3 monotone
and angle bounded operators \cite{Hara25}. It is at the core of
many important developments in the theory of monotone operators in
reflexive Banach spaces \cite{Atto18,Livre1,Buim25,Bura02,Chul96,%
Penn01,Rei79b,Simo08,Zei90B}. In particular, it is an essential
component of the famous Br\'ezis--Haraux theorem \cite{Brez76},
which studies how close the range of the sum of two monotone
operators in a Hilbert space is to the Minkowski sum of their
ranges. A closely related function, later
introduced independently in \cite{Fitz88}, is the 
\emph{Fitzpatrick function} 
\begin{equation}
\label{e:10} 
F_A=H_A+\pair{\cdot}{\cdot}. 
\end{equation} 
A basic property of the Haraux function is that, if $A$
is maximally monotone, then 
\begin{equation} 
\label{e:5}
H_A\geq 0\quad\text{and}\quad\gra A
=\menge{(x,u^*)\in\XX\times\XX^*}{H_A(x,u^*)=0}. 
\end{equation} 
This follows for instance from \cite[Corollary~3.9]{Fitz88} and
\eqref{e:10}. A natural question is whether the inequality 
$H_A\geq 0$ can be improved. A related question concerns the
Fenchel--Young inequality. Consider a proper function
$\varphi\colon\XX\to\RX$ with conjugate $\varphi^*$, and define 
the associated \emph{Fenchel--Young function} by
\begin{equation} 
\label{e:2}
L_\varphi\colon\XX\times\XX^*\to\RX\colon(x,u^*)\mapsto
\varphi(x)+\varphi^*(u^*)-\pair{x}{u^*}. 
\end{equation} 
As is well known \cite[Section~10]{More66},
\begin{equation}
\label{e:6}
L_\varphi\geq 0\quad\text{and}\quad\gra\partial\varphi
=\menge{(x,u^*)\in\XX\times\XX^*}{L_\varphi(x,u^*)=0},
\end{equation} 
and one may likewise ask whether
better lower bounds can be found. These two questions are not only
of theoretical interest but they also impact concrete applications
in areas such as inverse problems \cite{Andr25}, evolution
inclusions \cite{Atto18}, optimal transportation
\cite{Car23a,Car23b}, and machine learning \cite{Blon20,Rako24}. 

Fix $(x,u^*)\in\XX\times\XX^*$. Then, in view of \eqref{e:1}, given
$(y,y^*)\in\gra A$, the inequality
$H_A(x,u^*)\geq\pair{x-y}{y^*-u^*}$ furnishes a trivial lower
bound. To make it exploitable, $(y,y^*)$ should depend on $(x,u^*)$
in a suitable way. The first instance of such a lower bound we have
found in the literature is the following.

\begin{proposition}[\protect{\cite[Lemma~2.3]{Peno06}}] 
\label{p:11}
Let $A\colon\XX\to 2^{\XX^*}$ be maximally monotone, let $x\in\XX$,
let $u^*\in\XX^*$, and let $\Delta\colon\XX\to 2^{\XX^*}$ be the 
\emph{duality mapping} of $\XX$, i.e.,
\begin{equation}
\label{e:dm}
(\forall x\in\XX)\;\;\Delta(x)=
\menge{x^*\in\XX^*}{\|x\|^2=\pair{x}{x^*}=\|x^*\|^2}.
\end{equation}
Then there exists $(z,z^*)\in\gra A$ such that
$z^*-u^*\in\Delta(x-z)$ and $H_A(x,u^*)\geq\|x-z\|^2$.
\end{proposition}

If $\XX$ is a Hilbert space (identified with its dual), then
$\Delta=\Id$ and the inclusion $z^*-u^*\in\Delta(x-z)$ in
Proposition~\ref{p:11} becomes $x+u^*-z=z^*\in Az$, i.e.,
$z=(\Id+A)^{-1}(x+u^*)=J_{A}(x+u^*)$ since the \emph{resolvent}
$J_{A}=(\Id+A)^{-1}$ is single-valued by
\cite[Corollary~23.11(i)]{Livre1}. We then deduce at once from
Proposition~\ref{p:11} that 
\begin{equation}
\label{e:0}
H_A(x,u^*)\geq\|x-J_A(x+u^*)\|^2.
\end{equation}
This inequality appears explicitly in \cite[Equation~(1)]{Atto18},
where it is also derived from \cite[Lemma~2.3]{Peno06}. Unaware of
these results, Carlier recently proposed in \cite{Car23a} a
parametrized version of \eqref{e:0} along with a lower bound for
\eqref{e:2}, with elegant applications to convex analysis and
transportation theory (see also \cite{Baus23} for further results
in Hilbert spaces).

\begin{proposition}
\label{p:21}
Suppose that $\XX$ is a Hilbert space, let $x\in\XX$, let
$u^*\in\XX$, and let $\gamma\in\RPP$. Then the following hold:
\begin{enumerate}
\item
\label{p:21i}
{\em {\cite[Section~2]{Car23a}}}
Let $A\colon\XX\to 2^{\XX}$ be maximally monotone. Then
\begin{equation}
H_A(x,u^*)\geq
\dfrac{\bigl\|x-J_{\gamma A}(x+\gamma u^*)\bigr\|^2}{\gamma}.
\end{equation}
\item
\label{p:21ii}
{\em {\cite[Lemma~1.1]{Car23a}}}
Let $\varphi\colon\XX\to\RX$ be a proper lower semicontinuous
convex function and let $\prox_\varphi$ be its 
\emph{proximity operator},
i.e., $\prox_\varphi=J_{\partial\varphi}$. Then
\begin{equation}
L_\varphi(x,u^*)\geq
\dfrac{\bigl\|x-\prox_{\gamma\varphi}(x+\gamma u^*)\bigr\|^2}
{\gamma}.
\end{equation}
\end{enumerate}
\end{proposition}

Following \cite{Car23a}, an extension of
Proposition~\ref{p:21}\ref{p:21i} to reflexive real Banach spaces
was proposed in \cite{Bura25} in the following form (this result is
stated with the condition $\dom H_W=\XX\times\XX^*$ in
\cite{Bura25}, but the weaker, more checkable condition $\dom
W=\XX$ suffices; see Example~\ref{ex:22}). 

\begin{proposition}[\protect{\cite[Theorem~1]{Bura25}}]
\label{p:22}
Let $A\colon\XX\to 2^{\XX^*}$ and $W\colon\XX\to 2^{\XX^*}$ be
maximally monotone, let $x\in\XX$, let $u^*\in\XX^*$, and let
$\gamma\in\RPP$. Suppose that $W$ is
$\alpha$-strongly monotone for some $\alpha\in\RPP$, and
$\dom W=\XX$. Let $x^*\in Wx$ and set 
$z=(W+\gamma A)^{-1}(x^*+\gamma u^*)$. Then 
$H_A(x,u^*)\geq\alpha\|x-z\|^2/\gamma$.
\end{proposition}

The main objective of the present paper is to derive new lower
bounds on the Haraux and Fenchel--Young functions in reflexive real
Banach spaces under minimal assumptions on the set-valued operator
$A$ in \eqref{e:1} and the function $\varphi$ in \eqref{e:2},
respectively. As seen above, the basic inequality \eqref{e:0} in
Hilbert spaces can be deduced from Proposition~\ref{p:11}. 
It can also be deduced from Proposition~\ref{p:22} with $W=\Id$. 
However, these propositions yield in general different conclusions.
As we shall see, two notions of resolvent for set-valued operators
implicitly underlie these inequalities: metric resolvents in
Proposition~\ref{p:11}, and warped resolvents in
Proposition~\ref{p:22}. In the Hilbertian setting, metric
resolvents coincide with instances of warped resolvents
and we recover Proposition~\ref{p:21}. 

In Section~\ref{sec:2}, we introduce our notation. Lower bounds on
the Haraux function of general set-valued operators based on metric
resolvents are derived in Section~\ref{sec:3} and lower bounds on
the Fenchel--Young function of proper functions based on metric
proximity operators are derived in Section~\ref{sec:4}. In
Sections~\ref{sec:5} and \ref{sec:6}, we provide alternative bounds
based on warped resolvents for the Haraux function and on warped
proximity operators for the Fenchel--Young function. These bounds
are new even in the basic setting of maximally monotone operators
and lower semicontinuous convex functions in Euclidean spaces, in
which case we show that they can be more easily computable and
sharper than those produced by Proposition~\ref{p:21}.
Section~\ref{sec:7} proposes applications to composite monotone
inclusions and Section~\ref{sec:8} concludes the paper with some
potential directions for future work.

\section{Notation}
\label{sec:2}
Let $A\colon\XX\to 2^{\XX^*}$ and let
$A^{-1}\colon\XX^*\to 2^{\XX}\colon x^*\mapsto
\menge{x\in\XX}{x^*\in Ax}$ be its \emph{inverse}.
The \emph{domain} of $A$ is
$\dom A=\menge{x\in\XX}{Ax\neq\emp}$, the \emph{range} of $A$ is
$\ran A=\bigcup_{x\in\dom A}Ax$, and the set of zeros of $A$ is
$\zer A=\menge{x\in\XX}{0\in Ax}$. We say that $A$ is
\emph{monotone} if
\begin{equation}
\label{e:m1}
\big(\forall (x,x^*)\in\gra A\big)
\big(\forall (y,y^*)\in\gra A\big)
\quad\pair{x-y}{x^*-y^*}\geq 0
\end{equation}
and \emph{maximally monotone} if
\begin{equation}
\label{e:m2}
\brk1{\forall (x,x^*)\in\XX\times\XX^*}\\
\bigl[\:(x,x^*)\in\gra A\;\Leftrightarrow\;
\bigl(\forall (y,y^*)\in\gra A\bigr)\;\;
\pair{x-y}{x^*-y^*}\geq 0\:\bigr].
\end{equation}
Let $\phi\colon\RP\to\RPX$ be increasing and vanishing only at $0$.
Then $A$ is $\phi$-\emph{uniformly monotone} if
\begin{equation}
\label{e:u1}
\brk1{\forall (x,x^*)\in\gra A}
\brk1{\forall (y,y^*)\in\gra A}\quad
\pair{x-y}{x^*-y^*}\geq\phi\brk1{\|x-y\|}.
\end{equation}
In particular, if $\phi=\alpha|\cdot|^2$ for some
$\alpha\in\RPP$, then $A$ is $\alpha$-\emph{strongly monotone}, 
that is, 
\begin{equation}
\label{e:u2}
\brk1{\forall (x,x^*)\in\gra A}
\brk1{\forall (y,y^*)\in\gra A}\quad
\pair{x-y}{x^*-y^*}\geq\alpha\|x-y\|^2.
\end{equation}
If $A$ is monotone and $\dom A\times\ran A\subset\dom H_A$, then 
it is \emph{$3^*$ monotone}.
At last, $A$ is \emph{injective} if 
$(\forall x\in\XX)$ $(\forall y\in\XX)$
$Ax\cap Ay\neq\emp$ $\Rightarrow$ $x=y$.

Let $\varphi\colon\XX\to\RXX$. 
The \emph{domain} of $\varphi$ is 
$\dom\varphi=\menge{x\in\XX}{\varphi(x)<\pinf}$ and the
\emph{conjugate} of $\varphi$ is the function 
$\varphi^*\colon\XX^*\to\RXX\colon x^*\mapsto
\sup_{x\in\XX}\brk{\pair{x}{x^*}-\varphi(x)}$. Further, $\varphi$
is \emph{cofinite} if $\dom\varphi^*=\XX^*$ and 
\emph{supercoercive} if 
$\lim_{\|x\|\to\pinf}\varphi(x)/\|x\|=\pinf$. Now suppose that
$\varphi\colon\XX\to\RX$ is \emph{proper}, i.e.,
$\dom\varphi\neq\emp$. The \emph{subdifferential} of $\varphi$ is
the operator
\begin{equation}
\label{e:subdiff}
\partial \varphi\colon\XX\to 2^{\XX^*}\colon x\mapsto
\menge{x^*\in\XX^*}{(\forall y\in\XX)\;
\pair{y-x}{x^*}+\varphi(x)\leq\varphi(y)}. 
\end{equation}
We denote by $\Gamma_0(\XX)$ the class of proper lower
semicontinuous convex functions from $\XX$ to $\RX$.
Let $f\in\Gamma_0(\XX)$. 
If $f$ is Gateaux differentiable on $\intdom f\neq\emp$, then the
\emph{Bregman distance} associated with $f$ is 
\begin{equation}
\label{e:B}
\begin{aligned}
D_f\colon\XX\times\XX&\to\,[0,\pinf]\\
(x,y)&\mapsto 
\begin{cases}
f(x)-f(y)-\pair{x-y}{\nabla f(y)},&\text{if}\;\;y\in\intdom f;\\
\pinf,&\text{otherwise}.
\end{cases}
\end{aligned}
\end{equation}
Finally, $f$ is a \emph{Legendre function} if it is 
essentially smooth in the sense that $\partial f$ is 
both locally bounded and single-valued on its
domain, and essentially strictly convex in the sense
that $\partial f^*$ is locally bounded on its domain and 
$f$ is strictly convex on every convex subset of $\dom\partial f$
\cite{Ccma01}.

\section{Lower bounds on the Haraux function based on metric
resolvents}
\label{sec:3} 

The Voisei--Z\u{a}linescu inequality \cite[Theorem~2.6]{Vois09}
gives a lower bound on the Haraux function of a maximally monotone
operator in terms of the distance to its graph. The main result of
this section is a sharpening of this inequality. As
in Proposition~\ref{p:11}, the duality mapping of \eqref{e:dm}
plays a central role via the following notion of a metric resolvent
(see \cite[Definition~3.4]{Peno01} and the references therein). 

\begin{definition}
\label{d:7}
Let $A\colon\XX\to2^{\XX^*}$ and let $\Delta$ be the duality
mapping of $\XX$. Then the \emph{metric resolvent} of $A$ is
\begin{equation}
\label{e:jpp}
R_A\colon\XX\to 2^{\XX}\colon x\mapsto
\menge{z\in\XX}{0\in Az+\Delta(z-x)}.
\end{equation}
\end{definition}

Let us recall a few facts about duality mappings.

\begin{lemma}
\label{l:4}
Let $\Delta$ be the duality mapping of $\XX$. Then the following
hold:
\begin{enumerate}
\item
\label{l:4i}
{\em {\cite[Theorem~I.4.4]{Cior90}}}
$\Delta=\partial\|\cdot\|^2/2$.
\item
\label{l:4ii}
{\em {\cite[Corollary~V.2.6]{Cior90}}}
$\Delta$ is maximally monotone. 
\item
\label{l:4iii}
{\em {\cite[Proposition~I.4.7(c)]{Cior90}}}
Let $x\in\XX$. Then $\Delta(-x)=-\Delta(x)$.
\item
\label{l:4iv}
{\em {\cite[Theorem~II.1.8]{Cior90}}}
Suppose that $\XX$ is strictly convex. Then $\Delta$ is strictly
monotone. 
\end{enumerate}
\end{lemma}

\begin{example}
\label{ex:12}
Let $C$ be a nonempty closed convex subset of $\XX$, let $\proj_C$
be the \emph{metric projection} operator onto $C$, i.e.,
\begin{equation}
\proj_C\colon\XX\to 2^{C}\colon x\mapsto
\menge{z\in C}{(\forall y\in C)\;\;\|x-z\|\leq\|x-y\|},
\end{equation}
and let $N_C$ be the normal cone operator of $C$, i.e.,
\begin{equation} 
\label{e:normalcone}
N_C\colon\XX\to 2^{\XX^*}\colon x\mapsto
\begin{cases}
\Menge1{x^*\in\XX^*}{(\forall y\in C)\;
\pair{y-x}{x^*}\leq 0},&\text{if}\;\;x\in C;\\
\emp,&\text{otherwise.}
\end{cases}
\end{equation}
Then it follows from \cite[Remarque~8.1.5a]{Laur72},
Lemma~\ref{l:4}\ref{l:4i}, and Definition~\ref{d:7} that 
$\proj_C=R_{N_C}$.
\end{example}

\begin{notation}
\label{n:1}
Let $\gamma\in\RPP$. We define a norm by
\begin{equation}
\label{e:n}
|||\cdot|||_\gamma\colon\XX\times\XX^*\to\RP\colon(x,x^*)\mapsto
\sqrt{\|x\|^2/\gamma+\gamma\|x^*\|^2}
\end{equation}
and denote the associated distance function to a set 
$\boldsymbol{C}\subset\XX\times\XX^*$ by
\begin{equation}
\label{e:d}
d_{\boldsymbol{C},\gamma}\colon\XX\times\XX^*\to\RPX\colon
(x,x^*)\mapsto\inf_{(y,y^*)\in\boldsymbol{C}}
|||(x,x^*)-(y,y^*)|||_\gamma.
\end{equation}
\end{notation}

The proposed lower bounds are as follows.

\begin{proposition}
\label{p:29}
Let $A\colon\XX\to2^{\XX^*}$, let $x\in\XX$, let $u^*\in\XX^*$, and
let $\gamma\in\RPP$. Suppose that $z\in R_{\gamma(A-u^*)}x$. Then,
using Notation~\ref{n:1}, 
\begin{equation}
\label{e:2025}
\displaystyle{H_A(x,u^*)\geq\dfrac{\|x-z\|^2}{\gamma}
\geq\dfrac{1}{2}d^2_{\gra A,\gamma}(x,u^*)}.
\end{equation}
\end{proposition}
\begin{proof}
Since $z\in R_{\gamma(A-u^*)}x$, it follows from \eqref{e:jpp}
that there exists $z^*\in Az$ such that
$0\in\gamma(z^*-u^*)+\Delta(z-x)$. Consequently, $(z,z^*)\in\gra A$
and, appealing to Lemma~\ref{l:4}\ref{l:4iii}, we obtain
\begin{equation}
\label{e:md}
\gamma(z^*-u^*)\in\Delta(x-z).
\end{equation}
In turn, \eqref{e:1} and \eqref{e:dm} yield
\begin{equation}
H_A(x,u^*)\geq\pair{x-z}{z^*-u^*}=\dfrac{
\pair{x-z}{\gamma(z^*-u^*)}}{\gamma}=\dfrac{\|x-z\|^2}{\gamma}.
\end{equation}
Next, we derive from \eqref{e:md}, \eqref{e:dm}, and \eqref{e:d}
that
\begin{align}
\dfrac{\|x-z\|^2}{\gamma}
&=\dfrac{\|x-z\|^2+\gamma^2\|u^*-z^*\|^2}{2\gamma}
\nonumber\\
&=\dfrac{1}{2}|||(x,u^*)-(z,z^*)|||_\gamma^2
\nonumber\\
&\geq\dfrac{1}{2}\inf_{(y,y^*)\in\gra A}
|||(x,u^*)-(y,y^*)|||_\gamma^2
\nonumber\\
&=\dfrac{1}{2}d_{\gra A,\gamma}^2(x,u^*),
\end{align}
which yields the rightmost inequality.
\end{proof}

\begin{theorem}
\label{t:30}
Let $A\colon\XX\to2^{\XX^*}$ be maximally monotone, let $x\in\XX$, 
let $u^*\in\XX^*$, and let $\gamma\in\RPP$. Then, using
Notation~\ref{n:1}, the following hold:
\begin{enumerate}
\item
\label{t:30i}
$R_{\gamma(A-u^*)}x\neq\emp$.
\item
\label{t:30ii}
Let $z\in R_{\gamma(A-u^*)}x$. Then
\begin{equation}
\label{e:30}
\displaystyle{H_A(x,u^*)\geq\dfrac{\|x-z\|^2}{\gamma}
\geq\dfrac{1}{2}d^2_{\gra A,\gamma}(x,u^*)}.
\end{equation}
\item
\label{t:30iii}
Suppose that $\XX$ is strictly convex. Then $R_{\gamma(A-u^*)}$ is
single-valued and 
\begin{equation}
\label{e:30'}
\displaystyle{H_A(x,u^*)\geq
\dfrac{\|x-R_{\gamma(A-u^*)}x\|^2}{\gamma}
\geq\dfrac{1}{2}d^2_{\gra A,\gamma}(x,u^*)}.
\end{equation}
\end{enumerate}
\end{theorem}
\begin{proof}
\ref{t:30i}:
By \cite[Theorem~10.6]{Simo98}, 
$(x,\gamma u^*)\in\XX\times\XX^*=\gra(\gamma A)+\gra(-\Delta)$. 
Therefore, there exists $(z,z^*)\in\gra A$ such that
$(x-z,\gamma u^*-\gamma z^*)\in\gra(-\Delta)$, hence
$\gamma(z^*-u^*)\in\Delta(x-z)$, which implies that
$0\in\gamma(z^*-u^*)+\Delta(z-x)\subset
\gamma(A-u^*)z+\Delta(z-x)$. In view of Definition~\ref{d:7},
$z\in R_{\gamma(A-u^*)}x$.

\ref{t:30ii}: This follows from \ref{t:30i} and
Proposition~\ref{p:29}.

\ref{t:30iii}: Taking \ref{t:30ii} into account, it is enough to
show that $R_{\gamma(A-u^*)}x$ contains at most one point. Suppose
that $\{z_1,z_2\}\subset R_{\gamma(A-u^*)}x$. Then we infer from
Definition~\ref{d:7} that there exist $w^*_1\in\Delta(z_1-x)$ and 
$w^*_2\in\Delta(z_2-x)$ such that $(z_1,-w^*_1)\in\gra A$ and 
$(z_2,-w^*_2)\in\gra A$. By monotonicity of $A$, 
\begin{equation}
\pair{(z_1-x)-(z_2-x)}{w^*_1-w^*_2}=
\pair{z_1-z_2}{w^*_1-w^*_2}\leq 0.
\end{equation}
However, Lemma~\ref{l:4}\ref{l:4ii} forces
\begin{equation}
\pair{(z_1-x)-(z_2-x)}{w^*_1-w^*_2}\geq 0.
\end{equation}
Thus, $\pair{(z_1-x)-(z_2-x)}{w^*_1-w^*_2}=0$ and, since 
Lemma~\ref{l:4}\ref{l:4iv} asserts that $\Delta$ is strictly
monotone, we conclude that $z_1-x=z_2-x$, i.e., that $z_1=z_2$.
\end{proof}

\begin{remark}
\label{r:23}
Theorem~\ref{t:30} extends existing results as follows:
\begin{enumerate}
\item
\label{r:23i}
For $\gamma=1$, it follows in particular from 
Theorem~\ref{t:30}\ref{t:30ii} that, if $z\in R_{A-u^*}x$, then 
$H_A(x,u^*)\geq\|x-z\|^2$. This conclusion is identical to that of 
\cite[Lemma~2.3]{Peno06} (see Proposition~\ref{p:11}). Indeed,
arguing as in the proof of Theorem~\ref{t:30}\ref{t:30i}, the
inclusion $z\in R_{A-u^*}x$ secures the existence of 
$(z,z^*)\in\gra A$ such that $z^*-u^*\in\Delta(x-z)$.
\item
\label{r:23ii}
Regarding Theorem~\ref{t:30}\ref{t:30ii}, the smaller lower bound
\begin{equation}
\label{e:14}
H_A(x,u^*)\geq\dfrac{1}{4}d_{\gra A,1}^2(x,u^*)
\end{equation}
was established in \cite[Theorem~2.6]{Vois09} using more technical
arguments. Our simple proof improves it by a factor 2, and for any
$\gamma\in\RPP$.
\item
\label{r:23iii}
Regarding Theorem~\ref{t:30}\ref{t:30ii}, in the special case when
$\XX$ is Hilbertian and $\gamma=1$, the inequality
\begin{equation}
\label{e:39}
H_A(x,u^*)\geq\dfrac{1}{2}d_{\gra A,1}^2(x,u^*)
\end{equation}
appears in \cite[Theorem~4]{Bura25}.
\item
\label{r:23iv}
As pointed out in \cite[Exercise~9.7.11]{Borw10}, the constant
$1/2$ in Theorem~\ref{t:30}\ref{t:30ii} is the best possible to the
extent that, if $\XX$ is Hilbertian, $A=\Id$, 
and $\gamma=1$, then the inequalities become equalities.
\item
\label{r:23v}
Suppose that $\XX$ is Hilbertian. Then $\Delta=\Id$ and it follows
from Definition~\ref{d:7} that $R_{\gamma A}$ coincides with the
usual resolvent $J_{\gamma A}=(\Id+\gamma A)^{-1}$ of
\cite[Definition~23.1]{Livre1}. In this context, the inequality
$H_A(x,u^*)\geq\|x-J_{\gamma(A-u^*)}x\|^2/\gamma$ provided by 
Theorem~\ref{t:30}\ref{t:30iii} is easily seen to be equivalent to
the inequality $H_A(x,u^*)\geq\bigl\|x-J_{\gamma A}
(x+\gamma u^*)\bigr\|^2/\gamma$, which is precisely that
established in \cite[Section~2]{Car23a} (see
Proposition~\ref{p:21}\ref{p:21i}).
\end{enumerate}
\end{remark}

\section{Lower bounds on the Fenchel--Young function based on 
metric proximity operators}
\label{sec:4} 

We turn our attention to the case when $A$ is a subdifferential
operator to obtain lower bounds on the Fenchel--Young function of
\eqref{e:2}. The strategy is to exploit the conclusions of
Section~\ref{sec:3} via the following inequality which, in the case
of Hilbert spaces, appears in \cite[p.~167]{Brez76} (if $\varphi$
is convex, it is also a consequence of \cite[Theorem~3.7]{Fitz88}
and \eqref{e:10}). For completeness, we record a proof in our
context.

\begin{lemma}
\label{l:3}
Let $\varphi\colon\XX\to\RX$ be proper. 
Then $L_\varphi\geq H_{\partial\varphi}$. 
\end{lemma}
\begin{proof}
Let $(x,u^*)\in\XX\times\XX^*$ and 
$(y,y^*)\in\gra\partial\varphi$. Then, since $y\in\dom\varphi$,
\eqref{e:subdiff} yields
\begin{align}
\varphi(x)+\varphi^*(u^*)-\pair{x}{u^*}
&\geq\varphi(x)+\pair{y}{u^*}-\varphi(y)-\pair{x}{u^*}
\nonumber\\
&\geq\pair{x-y}{y^*}+\varphi(y)+\pair{y}{u^*}-\varphi(y)-
\pair{x}{u^*}\nonumber\\
&=\pair{x-y}{y^*-u^*}.
\end{align}
In view of \eqref{e:1} and \eqref{e:2}, this shows that
$L_\varphi(x,u^*)\geq H_{\partial\varphi}(x,u^*)$.
\end{proof}

Combining Proposition~\ref{p:29} and Lemma~\ref{l:3} yields the
following set of inequalities in terms of the metric resolvent of
Definition~\ref{d:7}.

\begin{proposition}
\label{p:24}
Let $\varphi\colon\XX\to\RX$ be proper, let $x\in\XX$, let
$u^*\in\XX^*$, and let $\gamma\in\RPP$. Suppose that 
$z\in R_{\partial(\gamma(\varphi-u^*))}x$. Then,
using Notation~\ref{n:1}, 
\begin{equation}
\label{e:202}
\displaystyle{
L_\varphi(x,u^*)\geq H_{\partial\varphi}(x,u^*)
\geq\dfrac{\|x-z\|^2}{\gamma}
\geq\dfrac{1}{2}d^2_{\gra\partial\varphi,\gamma}(x,u^*)}.
\end{equation}
\end{proposition}

To refine these inequalities, let us recall the notion of a metric
proximity operator, first introduced by Moreau \cite{Mor62b} in
Hilbert spaces.

\begin{definition}
\label{d:8}
Let $\varphi\in\Gamma_0(\XX)$. Then the \emph{metric proximity
operator} of $f$ is
\begin{equation}
\label{e:jjm1}
\prox_\varphi\colon\XX\to 2^{\XX}\colon x\mapsto
\underset{y\in\XX}{\Argmin}\:
\brk2{\varphi(y)+\dfrac{1}{2}\|x-y\|^2}.
\end{equation}
\end{definition}

The following theorem establishes lower bounds on $L_\varphi$ for
$\varphi\in\Gamma_0(\XX)$.

\begin{theorem}
\label{t:32}
Let $\varphi\in\Gamma_0(\XX)$, let $x\in\XX$, 
let $u^*\in\XX^*$, and let $\gamma\in\RPP$. Then, using
Notation~\ref{n:1}, the following hold:
\begin{enumerate}
\item
\label{t:32i}
$\prox_{\gamma(\varphi-u^*)}x\neq\emp$.
\item
\label{t:32ii}
Let $z\in\prox_{\gamma(\varphi-u^*)}x$. Then
\begin{equation}
\label{e:32}
\displaystyle{
L_\varphi(x,u^*)\geq H_{\partial\varphi}(x,u^*)
\geq\dfrac{\|x-z\|^2}{\gamma}
\geq\dfrac{1}{2}d^2_{\gra\partial\varphi,\gamma}(x,u^*)}.
\end{equation}
\item
\label{t:32iii}
Suppose that $\XX$ is strictly convex. Then
$\prox_{\gamma(\varphi-u^*)}$ is single-valued and 
\begin{equation}
\label{e:32'}
\displaystyle{
L_\varphi(x,u^*)\geq H_{\partial\varphi}(x,u^*)
\geq\dfrac{\|x-\prox_{\gamma(\varphi-u^*)}x\|^2}{\gamma}
\geq\dfrac{1}{2}d^2_{\gra\partial\varphi,\gamma}(x,u^*)}.
\end{equation}
\end{enumerate}
\end{theorem}
\begin{proof}
First, \cite[Proposition~47.F(1)]{Zeid85} asserts that
$\partial\varphi$ is maximally monotone. Next, we derive from 
Definition~\ref{d:8}, Fermat's rule
\cite[Proposition~47.12]{Zeid85}, the subdifferential sum rule
\cite[Theorem~47.B]{{Zeid85}}, Lemma~\ref{l:4}\ref{l:4i}, and
Definition~\ref{d:7} that
\begin{align}
\prox_{\gamma(\varphi-u^*)}x
&=\underset{y\in\XX}{\Argmin}\:\brk2{\gamma\brk1{\varphi(y)
-\pair{y}{u^*}}+\dfrac{1}{2}\|x-y\|^2}
\nonumber\\
&=\zer\partial\brk2{\gamma\brk1{\varphi-u^*}
+\frac{1}{2}\|\cdot-x\|^2}
\nonumber\\
&=\zer\brk1{\gamma\brk1{\partial{\varphi}-u^*}+
\Delta\brk{\cdot-x}}
\nonumber\\
&=\Menge1{z\in\XX}{0\in{\gamma\brk{\partial{\varphi}-u^*}}z+
\Delta\brk{z-x}}
\nonumber\\
&=R_{\gamma\brk{\partial{\varphi}-u^*}}x.
\end{align}
The claims are therefore consequences of Theorem~\ref{t:30}, where
$A=\partial\varphi$, and Lemma~\ref{l:3}.
\end{proof}

\begin{remark}
\label{r:24}
Suppose that $\XX$ is Hilbertian. Then, as in 
Remark~\ref{r:23}\ref{r:23v}, the inequality
$L_\varphi(x,u^*)\geq\|x-\prox_{\gamma(\varphi-u^*)}x\|^2/\gamma$ 
from Theorem~\ref{t:32}\ref{t:32iii} is equivalent to
$L_\varphi(x,u^*)\geq\|x-\prox_{\gamma\varphi}
(x+\gamma u^*)\|^2/\gamma$, which is precisely the lower bound 
established in \cite[Lemma~1.1]{Car23a} (see
Proposition~\ref{p:21}\ref{p:21ii}).
\end{remark}

\section{Lower bounds on the Haraux function based on warped
resolvents} 
\label{sec:5}

We derive lower bounds on the Haraux function of a general
set-valued operator $A\colon\XX\to 2^{\XX^*}$ in terms of an
auxiliary set-valued operator $W\colon\XX\to 2^{\XX^*}$ through the
notion of a warped resolvent. The results are then specialized to
the case when $W$ is at most single-valued and, in particular, when
it is the gradient of a Legendre function, which gives rise to
lower bounds in terms of Bregman distances. Several examples
illustrate these new bounds and an application to the asymptotic
behavior of families of set-valued operators is provided. Our
analysis relies on the following notion of a warped resolvent,
which was introduced in the case of at-most single-valued kernels
in \cite[Definition~1.1]{Jmaa20}.

\begin{definition}
\label{d:wr}
Let $A\colon\XX\to 2^{\XX^*}$ and $K\colon\XX\to 2^{\XX^*}$. Then 
the \emph{warped resolvent} of $A$ with kernel $K$ is
$J_A^K=(K+A)^{-1}\circ K$. 
\end{definition}

\begin{proposition}
\label{p:12}
Let $A\colon\XX\to2^{\XX^*}$, let $W\colon\XX\to2^{\XX^*}$, 
let $x\in\XX$, let $u^*\in\XX^*$, and let $\gamma\in\RPP$. Then 
\begin{equation}
\label{e:88}
J_{\gamma(A-u^*)}^Wx=(W+\gamma A)^{-1}(Wx+\gamma u^*).
\end{equation}
\end{proposition}
\begin{proof}
Let $z\in\XX$. Then 
\begin{eqnarray}
z\in J_{\gamma(A-u^*)}^Wx
&\Leftrightarrow&
z\in\brk1{W+\gamma(A-u^*)}^{-1}(Wx)\nonumber\\
&\Leftrightarrow&
(\exi x^*\in Wx)\quad z\in\brk1{W+\gamma(A-u^*)}^{-1}x^*\nonumber\\
&\Leftrightarrow&
(\exi x^*\in Wx)\quad x^*+\gamma u^*\in Wz+\gamma Az\nonumber\\
&\Leftrightarrow&
(\exi x^*\in Wx)\quad z\in(W+\gamma A)^{-1}(x^*+\gamma u^*)
\nonumber\\
&\Leftrightarrow&
z\in(W+\gamma A)^{-1}(Wx+\gamma u^*),
\end{eqnarray}
which proves \eqref{e:88}.
\end{proof}

\begin{proposition}
\label{p:6}
Let $A\colon\XX\to 2^{\XX^*}$, let $W\colon\XX\to 2^{\XX^*}$,
let $u^*\in\XX^*$, and let $\gamma\in\RPP$. Suppose that 
$x\in\dom W$ and that $z\in J_{\gamma(A-u^*)}^Wx$. Then the
following hold:
\begin{enumerate}
\item
\label{p:6i}
There exist $x^*\in Wx$ and $z^*\in Wz$ such that 
$H_A(x,u^*)\geq{\pair{x-z}{x^*-z^*}}/{\gamma}$.
\item
\label{p:6ii}
Suppose that $W$ is $\phi$-uniformly monotone. Then 
$H_A(x,u^*)\geq{\phi(\|x-z\|)}/{\gamma}$.
\item
\label{p:6iii}
Suppose that $W$ is $\alpha$-strongly monotone. Then 
$H_A(x,u^*)\geq{\alpha\|x-z\|^2}/{\gamma}$.
\item
\label{p:6iv}
Suppose that $W$ is at most single-valued. Then
$H_A(x,u^*)\geq{\pair{x-z}{Wx-Wz}}/{\gamma}$.
\item
\label{p:6v}
Let $f\in\Gamma_0(\XX)$ be Gateaux differentiable on 
$\dom\nabla f=\intdom f$ and suppose that $W=\nabla f$. Then 
\begin{equation}
H_A(x,u^*)\geq\dfrac{D_f(x,z)+D_f(z,x)}{\gamma}.
\end{equation}
\end{enumerate}
\end{proposition}
\begin{proof}
\ref{p:6i}:
Since Proposition~\ref{p:12} asserts that
$z\in J_{\gamma(A-u^*)}^Wx=(W+\gamma A)^{-1}(Wx+\gamma u^*)$, 
there exists $x^*\in Wx$ such that 
$x^*+\gamma u^*\in Wz+\gamma Az$, from which we
deduce that $z\in\dom W\cap\dom A$ and that there exists $z^*\in
Wz$ such that $x^*+\gamma u^*\in z^*+\gamma Az$. In turn,
\begin{equation}
\label{e:51}
\brk3{z,\frac{x^*-z^*}{\gamma}+u^*}\in\gra A.
\end{equation}
Consequently,
\begin{equation}
H_A(x,u^*)=\sup_{(y,y^*)\in\gra A}\pair{x-y}{y^*-u^*}
\geq\pair2{x-z}{\frac{x^*-z^*}{\gamma}+u^*-u^*}
=\dfrac{\pair{x-z}{x^*-z^*}}{\gamma}.
\end{equation}

\ref{p:6ii}: This follows from \ref{p:6i} and \eqref{e:u1}.

\ref{p:6iii}: Take $\phi=\alpha|\cdot|^2$ in \ref{p:6ii}.

\ref{p:6iv}: An immediate consequence of \ref{p:6i}.

\ref{p:6v}: This follows from \ref{p:6iv} and \eqref{e:B}.
\end{proof}

Let us characterize the situation in which the point $z$ 
in Proposition~\ref{p:6}\ref{p:6iv} is $x$ itself .

\begin{proposition}
\label{p:7}
In the setting of Proposition~\ref{p:6}\ref{p:6iv}, consider the
following statements:
\begin{enumerate}[label={\normalfont[\alph*]}]
\item
\label{p:7a}
$x\in J_{\gamma(A-u^*)}^Wx$. 
\item
\label{p:7b}
$(x,u^*)\in\gra A$.
\item
\label{p:7c}
$H_A(x,u^*)=0$.
\end{enumerate}
Then the following hold:
\begin{enumerate}
\item
\label{p:7i}
\ref{p:7a}$\Leftrightarrow$\ref{p:7b}.
\item
\label{p:7ii}
Suppose that $A$ is monotone. Then 
\ref{p:7b}$\Rightarrow$\ref{p:7c}.
\item
\label{p:7iii}
Suppose that $A$ is maximally monotone. Then 
\ref{p:7b}$\Leftrightarrow$\ref{p:7c}.
\end{enumerate}
\end{proposition}
\begin{proof}
\ref{p:7i}: We have
$x\in J_{\gamma(A-u^*)}^Wx$ $\Leftrightarrow$
$x\in (W+\gamma(A-u^*))^{-1}(Wx)$ $\Leftrightarrow$
$Wx\in Wx+\gamma(Ax-u^*)$ $\Leftrightarrow$ $u^*\in Ax$.

\ref{p:7ii}: By monotonicity,
$(\forall(y,y^*)\in\gra A)$ $\pair{x-y}{y^*-u^*}\leq 0$. Therefore,
$H_A(x,u^*)\leq 0$. At the same time, 
$H_A(x,u^*)=\sup_{(y,y^*)\in\gra A}\pair{x-y}{y^*-u^*}
\geq\pair{x-x}{u^*-u^*}=0$. 

\ref{p:7iii}: See \eqref{e:5}.
\end{proof}

The following result ensures the existence of the point $z$ in
Proposition~\ref{p:6}.

\begin{proposition}
\label{p:3}
Let $A\colon\XX\to 2^{\XX^*}$, let $W\colon\XX\to 2^{\XX^*}$,
let $u^*\in\XX^*$, and let $\gamma\in\RPP$. Suppose that 
$x\in\dom W$ and that one of the following holds:
\begin{enumerate}
\item
\label{p:3i}
There exists $x^*\in Wx$ such that 
$x^*+\gamma u^*\in\ran(W+\gamma A)$.
\item
\label{p:3ii}
$W+\gamma A$ is surjective.
\item
\label{p:3iii}
$W+\gamma A$ is maximally monotone and $(W+\gamma A)^{-1}$ is
locally bounded at every point in $\XX^*$.
\item
\label{p:3iv}
$W+\gamma A$ is maximally monotone and one of the following is
satisfied:
\begin{enumerate}
\item
$\dom W\cap\dom A$ is bounded. 
\item
$\dom W\cap\dom A$ is unbounded and 
\end{enumerate}
\begin{equation}
\lim_{\substack{y\in\dom W\cap\dom A\\ \|y\|\to\pinf}}\;
\brk3{\displaystyle{\inf_{y^*\in Wy+\gamma Ay}}\|y^*\|}=\pinf.
\end{equation}
\item
\label{p:3v}
$W+\gamma A$ is maximally monotone and $\phi$-uniformly monotone
with $\phi(t)/t\to\pinf$ as $t\to\pinf$.
\item
\label{p:3vi}
$W+\gamma A$ is maximally monotone and $\alpha$-strongly monotone
for some $\alpha\in\RPP$.
\item
\label{p:3vii}
$W$ and $A$ are monotone, $W+\gamma A$ is maximally monotone,
$\ran W+\gamma\ran A=\XX^*$, and one of the following is satisfied:
\begin{enumerate}
\item
\label{p:3viia}
$W$ and $A$ are $3^*$ monotone. 
\item
\label{p:3viib}
$W$ is $3^*$ monotone and $\dom A\subset\dom W$.
\item
\label{p:3viic}
$A$ is $3^*$ monotone and $\dom W\subset\dom A$.
\end{enumerate}
\item
\label{p:3ix}
$W=\partial f$, with $f\in\Gamma_0(\XX)$, $A$ is monotone,
$\partial f+\gamma A$ is maximally monotone,
$\dom\partial f^*+\gamma\ran A=\XX^*$, and either
$\dom A\subset\dom\partial f$ or $A$ is $3^*$ monotone.
\item
\label{p:3viii}
$W$ and $A$ are monotone, $W+\gamma A$ is maximally monotone, 
and one of the following is satisfied:
\begin{enumerate}
\item
\label{p:3viiia}
$W$ is $3^*$ monotone and surjective.
\item
\label{p:3viiib}
$A$ is $3^*$ monotone and surjective.
\end{enumerate}
\end{enumerate}
Then $J_{\gamma(A-u^*)}^Wx\neq\emp$.
\end{proposition}
\begin{proof}
For convenience, we set $M=W+\gamma A$. 

\ref{p:3i}:
Since $x^*+\gamma u^*\in\ran M=\dom M^{-1}$,
we have $M^{-1}(x^*+\gamma u^*)\neq\emp$. Therefore, 
by Proposition~\ref{p:12}, 
$J_{\gamma(A-u^*)}^Wx=M^{-1}(Wx+\gamma u^*)\neq\emp$.

\ref{p:3ii}$\Rightarrow$\ref{p:3i}: Indeed, $\ran M=\XX^*$.

\ref{p:3iii}$\Rightarrow$\ref{p:3ii}: 
This follows from the Br\'ezis--Browder surjectivity theorem 
\cite[Theorem~32.G]{Zei90B}.

\ref{p:3iv}$\Rightarrow$\ref{p:3ii}: 
See \cite[Corollary~32.35]{Zei90B}.

\ref{p:3v}$\Rightarrow$\ref{p:3iv}: 
In view of \ref{p:3iv}, we assume that $\dom M$ is unbounded and
show that 
\begin{equation}
\label{e:9}
\lim_{\substack{y\in\dom M\\ \|y\|\to\pinf}}\;
\brk3{\displaystyle{\inf_{y^*\in My}}\|y^*\|}=\pinf.
\end{equation}
For this purpose, fix $(v,v^*)\in\gra M$. Then \eqref{e:u1} yields
\begin{equation}
\brk1{\forall (y,y^*)\in\gra M}\quad
\|y-v\|\,\|y^*-v^*\|\geq\pair{y-v}{y^*-v^*}\geq\phi\brk1{\|y-v\|}.
\end{equation}
Therefore
\begin{equation}
\label{e:33}
\brk1{\forall (y,y^*)\in\gra M}\quad y\neq v\;\Rightarrow\;
\|y^*\|
\geq\|y^*-v^*\|-\|v^*\|
\geq\dfrac{\phi\brk1{\|y-v\|}}{\|y-v\|}-\|v^*\|.
\end{equation}
Since the right-hand side goes to $\pinf $ as $\|y\|\to\pinf$, we
obtain \eqref{e:9}.

\ref{p:3vi}$\Rightarrow$\ref{p:3v}: Take $\phi=\alpha|\cdot|^2$.

\ref{p:3vii}$\Rightarrow$\ref{p:3ii}: It follows from the
assumptions and the reflexive Banach space version
\cite[Theorem~2.2]{Rei79b} of the Br\'ezis--Haraux theorem
\cite[Th\'eor\`emes~3 et 4]{Brez76} that 
$\inte\ran M=\inte(\ran W+\gamma\ran A)$, which 
implies that $\ran M=\XX^*$.

\ref{p:3ix}$\Rightarrow$\ref{p:3vii}: Indeed, 
$W$ is $3^*$ monotone by \cite[Proposition~32.42]{Zei90B} and
$\ran W=\ran\partial f=\dom(\partial f)^{-1}=\dom\partial f^*$ 
by \cite[Theorem~51.A(ii)]{Zeid85}.

\ref{p:3viii}$\Rightarrow$\ref{p:3ii}: See
\cite[Corollary~11 and Remark~9(iv)]{Buim25}.
\end{proof}

\begin{example}
\label{ex:23}
Let $A\colon\XX\to 2^{\XX^*}$ and $W\colon\XX\to 2^{\XX^*}$
be maximally monotone, let $x\in\dom W$, let $u^*\in\XX^*$, and 
let $\gamma\in\RPP$. Suppose that the cone generated by 
$\dom W-\dom A$ is a closed vector subspace of $\XX$ and that
$W$ is $\phi$-uniformly monotone with $\phi(t)/t\to\pinf$ as 
$t\to\pinf$. Let $z\in J_{\gamma(A-u^*)}^Wx$. Then 
$H_A(x,u^*)\geq\phi(\|x-z\|)/{\gamma}$.
\end{example}
\begin{proof}
It follows from the Attouch--Riahi--Th\'era theorem 
\cite[Corollary~32.3]{Simo08} that $W+\gamma A$ is maximally
monotone. Additionally, $W+\gamma A$ is $\phi$-uniformly monotone
and therefore Proposition~\ref{p:3}\ref{p:3v} guarantees that
$J_{\gamma(A-u^*)}^Wx\neq\emp$. The conclusion therefore follows
from Proposition~\ref{p:6}\ref{p:6ii}.
\end{proof}

An important consequence of Propositions~\ref{p:6}\ref{p:6iv} and
\ref{p:3} is the following.

\begin{theorem}
\label{t:1}
Let $A\colon\XX\to 2^{\XX^*}$, let $\emp\neq\DD\subset\XX$, let 
$W\colon\DD\to\XX^*$, let $x\in\DD$, let $u^*\in\XX^*$, and let
$\gamma\in\RPP$. Suppose that $\DD\cap\dom A\neq\emp$ and that one
of properties \ref{p:3i}--\ref{p:3viii} in Proposition~\ref{p:3}
is satisfied, together with one of the following:
\begin{enumerate}
\item
\label{t:1i}
$W+\gamma A$ is injective.
\item
\label{t:1ii}
$W+\gamma A$ is strictly monotone.
\item
\label{t:1iii}
$W$ is uniformly monotone and $A$ is monotone.
\end{enumerate}
Set $z=J_{\gamma(A-u^*)}^Wx$. Then 
$H_A(x,u^*)\geq{\pair{x-z}{Wx-Wz}}/{\gamma}$.
\end{theorem}
\begin{proof}
Set $M=W+\gamma A$. In view of Propositions~\ref{p:6}\ref{p:6iv},
\ref{p:3}, and \ref{p:12}, it remains to show that the set
$M^{-1}(Wx+\gamma u^*)=J_{\gamma(A-u^*)}^Wx$ is a
singleton. 

\ref{t:1i}:
Let $(x^*,x_1)$ and $(x^*,x_2)$ be points in $\gra M^{-1}$. Then
$x^*\in Mx_1\cap Mx_2$ and therefore, by injectivity, $x_1=x_2$. 
Thus, $M^{-1}$ is at most single-valued.

\ref{t:1ii}$\Rightarrow$\ref{t:1i}:
Let $(x_1,x_2)\in\XX\times\XX$ be such that there exists
$x^*\in Mx_1\cap Mx_2$. Then $(x_1,x^*)$ and $(x_2,x^*)$ lie in
$\gra M$ and $\pair{x_1-x_2}{x^*-x^*}=0$. The strict monotonicity
of $M$ then forces $x_1=x_2$.

\ref{t:1iii}$\Rightarrow$\ref{t:1ii}:
Since $W$ is strictly monotone and $\gamma A$ is monotone, 
$W+\gamma A$ is strictly monotone.
\end{proof}

\begin{example}
\label{ex:24}
Let $A\colon\XX\to 2^{\XX^*}$, let $\emp\neq\DD\subset\XX$,
let $W\colon\DD\to\XX^*$, let $x\in\DD$, let $u^*\in\XX^*$, and 
let $\gamma\in\RPP$. Suppose that $A$ and $W$ are maximally
monotone, that the cone generated by $\DD-\dom A$ is a closed vector
subspace of $\XX$, and that $W$ is $\phi$-uniformly monotone with 
$\phi(t)/t\to\pinf$ as $t\to\pinf$. Set $z=J_{\gamma(A-u^*)}^Wx$.
Then $H_A(x,u^*)\geq\phi(\|x-z\|)/{\gamma}$.
\end{example}
\begin{proof}
The Attouch--Riahi--Th\'era theorem \cite[Corollary~32.3]{Simo08} 
guarantees that $W+\gamma A$ is maximally monotone. Since it is
also $\phi$-uniformly monotone, condition \ref{p:3v} of
Proposition~\ref{p:3} is satisfied. The conclusion therefore
follows from Theorem~\ref{t:1}\ref{t:1iii} and \eqref{e:u1}.
\end{proof}

Our next result establishes a lower bound in terms of Bregman
distances. Here, $W$ is the gradient of a Legendre function 
$f\in\Gamma_0(\XX)$ and the warped resolvent $J^{\nabla f}_{A}$
becomes the Bregman resolvent of $A$ studied in
\cite[Section~3.3]{Sico03}.

\begin{proposition}
\label{p:2}
Let $A\colon\XX\to 2^{\XX^*}$ be maximally monotone, 
let $f\in\Gamma_0(\XX)$ be a Legendre function, 
let $x\in\intdom f$, let $u^*\in\XX^*$, and let $\gamma\in\RPP$. 
Suppose that $(\intdom f)\cap\dom A\neq\emp$ and that one of the 
following holds:
\begin{enumerate}
\item
\label{p:2i}
$\gamma u^*+\intdom f^*\subset\ran(\nabla f+\gamma A)$.
\item
\label{p:2ii}
$(\intdom f^*)+\gamma\ran A=\XX^*$ and one of the following is
satisfied:
\begin{enumerate}
\item
$A$ is $3^*$ monotone.
\item
\label{p:2iib}
$\dom A\subset\intdom f$.
\end{enumerate}
\end{enumerate}
Set $z=J_{\gamma(A-u^*)}^{\nabla f}x$. Then
\begin{equation}
H_A(x,u^*)\geq\dfrac{D_f(x,z)+D_f(z,x)}{\gamma}.
\end{equation}
\end{proposition}
\begin{proof}
We apply Theorem~\ref{t:1} with $W=\nabla f$ and $\DD=\dom W$. 
First, since $f$ is essentially smooth, \cite[Theorem~5.6]{Ccma01}
asserts that $\DD=\intdom f$. In addition, since $W$ is maximally 
monotone and 
$(\intdom W)\cap\dom A=(\intdom f)\cap\dom A\neq\emp$, the
Rockafellar sum theorem \cite[Theorem~32.I]{Zei90B} asserts that 
$W+\gamma A$ is maximally monotone. 
Moreover, since $f$ is essentially strictly convex, it is
strictly convex on the convex set $\DD=\intdom f$, which makes
$W$ strictly monotone \cite[Proposition~25.10]{Zei90B}. In turn,
$W+\gamma A$ is
strictly monotone. This shows that property \ref{t:1ii} in
Theorem~\ref{t:1} is satisfied. On the other hand, since 
\cite[Theorem~5.10]{Ccma01} asserts that
$\ran\nabla f=\intdom f^*$,
property \ref{p:2i} above implies property
\ref{p:3i} in Proposition~\ref{p:3}, while 
property \ref{p:2ii} above implies property
\ref{p:3vii} in Proposition~\ref{p:3} since
\cite[Theorem~5.9(ii)]{Ccma01} asserts that
$\dom\partial f^*=\dom\nabla f^*=\intdom f^*$, while
\cite[Proposition~32.42]{Zei90B} asserts that $W$ is $3^*$
monotone. Altogether, the conclusion follows from
Theorem~\ref{t:1} and \eqref{e:B}.
\end{proof}

\begin{remark}
A sufficient condition for the property 
$(\intdom f^*)+\gamma\ran A=\XX^*$ in
Proposition~\ref{p:2}\ref{p:2ii} to hold is that $f$ be
supercoercive, as this property implies that $f^*$ is cofinite
\cite[Theorem~3.4]{Ccma01}.
\end{remark}

We now recover the existing results of 
Propositions~\ref{p:21}\ref{p:21i} and \ref{p:22}.

\begin{example}
\label{ex:1}
Suppose that $\XX$ is a Hilbert space. Then we retrieve 
Proposition~\ref{p:21}\ref{p:21i} by applying
Proposition~\ref{p:2}\ref{p:2iib} with the Legendre function
$f=\|\cdot\|^2/2$, which satisfies $f^*=f$,
$\intdom f=\intdom f^*=\XX$, and 
$D_f\colon\XX\times\XX\to\RR\colon (u,v)\mapsto\|u-v\|^2/2$. 
Alternatively, we can apply Example~\ref{ex:24} with $W=\Id$, 
which satisfies $\dom W=\XX$ and is $\phi$-uniformly monotone with
$\phi=|\cdot|^2$.
\end{example}

\begin{example}
\label{ex:22}
Using Proposition~\ref{p:12}, we retrieve Proposition~\ref{p:22} 
as a special case of Example~\ref{ex:23}, where $\dom W=\XX$ and
$\phi=\alpha|\cdot|^2$.
\end{example}

The next example is an illustration of Proposition~\ref{p:2}
for a maximally monotone operator which is not a subdifferential.

\begin{example}
\label{ex:311}
Let $\XX=\RR^2$ be the standard Euclidean plane, let
$\beta\in\RPP$, and let $\psi\colon\RR\to\RR$ be a Legendre
function with a $\beta$-Lipschitzian derivative. Set 
\begin{equation}
A\colon\XX\to\XX\colon(\xi_1,\xi_2)\mapsto
\brk1{\beta\xi_1-\psi'(\xi_1)-\xi_2,\xi_1+\beta\xi_2-\psi'(\xi_2)}
\end{equation}
and consider the Legendre function 
\begin{equation}
f\colon\XX\to\RR\colon(\xi_1,\xi_2)\mapsto\psi(\xi_1)+\psi(\xi_2).
\end{equation}
As observed in \cite[Remark~3.4]{Joca16}, $A$ is maximally 
monotone but it is not the subdifferential of a convex 
function. Now let $x=(\xi_1,\xi_2)\in\XX$ and
$u^*=(\mu^*_1,\mu^*_2)\in\XX$. Since
$\ran(\nabla f+A)=\XX$, we derive from 
Propositions~\ref{p:2} and \ref{p:12} that
\begin{align}
z=(\zeta_1,\zeta_2)
&=J_{A-u^*}^{\nabla f}x\nonumber\\
&=(\nabla f+A)^{-1}\brk1{\nabla f(x)+u^*}
\nonumber\\
&=\brk3{\dfrac{\beta\brk1{\psi'(\xi_1)+\mu^*_1}+
\psi'(\xi_2)+\mu^*_2}{1+\beta^2}\:,\:
\dfrac{\beta\brk1{\psi'(\xi_2)+\mu^*_2}-\psi'(\xi_1)
-\mu^*_1}{1+\beta^2}}
\end{align}
is well defined and that 
$D_f(x,z)+D_f(z,x)=
(\xi_1-\zeta_1)\brk1{\psi'(\xi_1)-\psi'(\zeta_1)}
+(\xi_2-\zeta_2)\brk1{\psi'(\xi_2)-\psi'(\zeta_2)}$.
Therefore, Proposition~\ref{p:2} yields 
\begin{equation}
H_A(x,u^*)\geq
{(\xi_1-\zeta_1)\brk1{\psi'(\xi_1)-\psi'(\zeta_1)}
+(\xi_2-\zeta_2)\brk1{\psi'(\xi_2)-\psi'(\zeta_2)}}.
\end{equation}
Let us note that the lower bound of
Proposition~\ref{p:21}\ref{p:21i} on the Haraux function
would be harder to compute. 
\end{example}

Next, we consider an application to the asymptotic behavior of a
family of set-valued operators in terms of the warped resolvents of
Definition~\ref{d:wr}.

\begin{proposition}
\label{p:25}
Let $\mathfrak{A}=\menge{A_t\colon\XX\to2^{\XX^*}}{t\in\RP}$ be 
a family of operators, let $\emp\neq\DD\subset\XX$, 
let $W\colon\DD\to\XX^*$ be $\phi$-uniformly monotone, and 
let $\gamma\in\RPP$. Suppose that there exists an operator 
$A\colon\XX\to2^{\XX^*}$ such that 
\begin{equation}
\label{e:25}
\brk1{\forall(x,u^*)\in\gra A}\quad \lim_{t\to\pinf} 
H_{A_t}(x,u^*)=0.
\end{equation}
Additionally, suppose that, for every $B\in\mathfrak{A}\cup\{A\}$,
$\DD\cap\dom B\neq\emp$ and that one of properties
\ref{p:3ii}--\ref{p:3viii} in Proposition~\ref{p:3} is satisfied,
together with one of properties \ref{t:1i}--\ref{t:1iii} in 
Theorem~\ref{t:1}. Then the following hold:
\begin{enumerate}
\item 
\label{p:25i}
Let $y^*\in\XX^*$. Then 
$(W+\gamma A_t)^{-1}y^*\to (W+\gamma A)^{-1}y^*$.
\item
\label{p:25ii}
Let $y\in\DD$. Then $J_{\gamma A_t}^Wy\to J_{\gamma A}^Wy$.
\item
\label{p:25iii}
\emph{\cite[Proposition~2.1]{Atto18}}
Suppose that $\XX$ is Hilbertian, that
$W=\Id$, and that the operators 
$(A_t)_{t\in\RP}$  and $A$ are maximally monotone. 
Then $(A_t)_{t\in\RP}$ graph-converges to $A$. 
\end{enumerate}
\end{proposition}
\begin{proof}
Set $x=(W+\gamma A)^{-1}y^*$, which is well defined since 
$W+\gamma A$ is surjective and injective. Note that $y^*\in
Wx+\gamma Ax$ and set $u^*=\gamma^{-1}(y^*-Wx)$. Then 
$(x,u^*)\in\gra A$ and, in view of Proposition~\ref{p:12}, 
Theorem~\ref{t:1} asserts that, for every $t\in\RP$,
\begin{align}
\label{e:pll2}
H_{A_t}\brk{x,u^*}
&\geq\dfrac{1}{\gamma}\pair1{x-J_{\gamma(A_t-u^*)}^Wx}
{Wx-W\brk1{J_{\gamma(A_t-u^*)}^Wx}}
\nonumber\\
&=\dfrac{1}{\gamma}\pair1{x-(W+\gamma A_t)^{-1}
\brk{Wx+\gamma u^*}}{Wx-W\brk1{(W+\gamma A_t)^{-1}
\brk{Wx+\gamma u^*}}}
\nonumber\\
&\geq\dfrac{1}{\gamma}
\phi\brk3{\left\|x-(W+\gamma A_t)^{-1}\brk3{Wx+\gamma 
\brk2{\dfrac{y^*-Wx}{\gamma}}}\right\|}
\nonumber\\
&=\dfrac{1}{\gamma}\phi\brk2{
\bigl\|(W+\gamma A)^{-1}y^*-(W+\gamma A_t)^{-1}y^*\bigr\|}.
\end{align}

\ref{p:25i}: This follows from \eqref{e:25} and \eqref{e:pll2}.

\ref{p:25ii}: 
Set $y^*=Wy$ in \eqref{e:pll2}, and invoke \eqref{e:25} and
Definition~\ref{d:wr}. 

\ref{p:25iii}:
A direct consequence of \ref{p:25i} and
\cite[Proposition~3.60]{Atto84}.
\end{proof}

Outside of Hilbert spaces, the lower bounds on the Haraux function
produced in this section via warped resolvents are different from
those produced in Section~\ref{sec:3} via metric resolvents. In
particular, the lower bounds of Section~\ref{sec:3} involve the
distance to the graph of the operator. We show that it is possible
to obtain such bounds with warped resolvents based on a certain
type of kernel $W$. For this purpose, we introduce the following
property, which implies the strong monotonicity of both $W$ and its
inverse.

\begin{definition}
\label{d:10}
Let $\emp\neq\DD\subset\XX$, let $W\colon\DD\to\XX^*$, 
let $\alpha\in\RPP$, and let $\beta\in\RPP$. Then $W$ is 
\emph{$(\alpha,\beta)$-jointly strongly monotone} if 
\begin{equation}
\label{e:100}
(\forall x\in\DD)(\forall y\in\DD)
\quad\pair{x-y}{Wx-Wy}\geq\dfrac{\alpha}{2}\|x-y\|^2
+\dfrac{\beta}{2}\|Wx-Wy\|^2.
\end{equation}
\end{definition}

\begin{example}
\label{ex:31}
Let $\emp\neq\DD\subset\XX$, let $W\colon\DD\to\XX^*$, let 
$\alpha\in\RPP$, and let $\beta\in\RPP$. Then the following hold:
\begin{enumerate}
\item
\label{ex:31i} 
Suppose that $W$ is $\alpha$-strongly monotone. Then the 
following are satisfied:
\begin{enumerate}
\item
\label{ex:31ia}
Suppose that $W$ is $\beta$-cocoercive, i.e., 
\begin{equation}
(\forall x\in\DD)(\forall y\in\DD)\quad 
\pair{x-y}{Wx-Wy}\geq\beta\|Wx-Wy\|^2.
\end{equation}
Then $W$ is $(\alpha,\beta)$-jointly strongly monotone.
\item
\label{ex:31ib}
Suppose that $W$ is $\beta$-Lipschitzian. Then $W$ is
$(\alpha,\alpha/\beta^2)$-jointly strongly monotone.
\item
\label{ex:31ic}
Suppose that $W$ is nonexpansive. Then $W$ is 
$(\alpha,\alpha)$-jointly strongly monotone.
\end{enumerate}
\item
\label{ex:31ii} 
Let $f\colon\XX\to\RR$ be an $\alpha$-strongly convex,
Fr\'echet differentiable function. Suppose that 
$\nabla f$ is $\beta^{-1}$-Lipschitzian and that $W=\nabla f$. Then 
$\nabla f$ is $(\alpha,\beta)$-jointly strongly monotone.
\end{enumerate}
\end{example}
\begin{proof}
\ref{ex:31i}
Let $\{x,y\}\subset\DD$. Since $W$ is $\alpha$-strongly monotone,
we have $\pair{x-y}{Wx-Wy}\geq\alpha\|x-y\|^2$.

\ref{ex:31ia}:
Since $W$ is $\beta$-cocoercive, we have 
$\pair{x-y}{Wx-Wy}\geq\beta\|Wx-Wy\|^2$. Therefore, 
\begin{equation}
\pair{x-y}{Wx-Wy}\geq\max\{\alpha\|x-y\|^2,
\beta\|Wx-Wy\|^2\}\geq\dfrac{\alpha}{2}\|x-y\|^2
+\dfrac{\beta}{2}\|Wx-Wy\|^2.
\end{equation}

\ref{ex:31ib}:
Since $W$ is $\beta$-Lipschitzian, we get
\begin{equation}
\pair{x-y}{Wx-Wy}\geq
\dfrac{\alpha}{2}\|x-y\|^2+\dfrac{\alpha}{2}\|x-y\|^2
\geq\dfrac{\alpha}{2}\|x-y\|^2+\dfrac{\alpha}{2\beta^2}
\|Wx-Wy\|^2.
\end{equation}

\ref{ex:31ic}: Take $\beta=1$ in \ref{ex:31ib}.

\ref{ex:31ii}: Since $f$ is $\alpha$-strongly convex, 
$\nabla f$ is $\alpha$-strongly monotone. Additionally,
$\nabla f$ is $\beta$-cocoercive by the Baillon--Haddad 
theorem \cite[Corollaire~10]{Bail77}. We conclude by invoking 
\ref{ex:31ia}.
\end{proof}
Joint strong monotonicity allows us to derive from
Proposition~\ref{p:6} lower bounds similar to those of
Proposition~\ref{p:29}.

\begin{proposition}
\label{p:30}
Let $A\colon\XX\to2^{\XX^*}$, let $\emp\neq\DD\subset\XX$, 
let $W\colon\DD\to\XX^*$, 
let $x\in\DD$, let $u^*\in\XX^*$, let $\gamma\in\RPP$, 
let $\alpha\in\RPP$, and let $\beta\in\RPP$. 
Suppose that $z\in J^W_{\gamma(A-u^*)}x$ and that 
$W$ is $(\alpha,\beta)$-jointly strongly monotone. 
Then, using Notation~\ref{n:1},
\begin{equation}
H_A(x,u^*)\geq\dfrac{\pair{x-z}{Wx-Wz}}{\gamma}
\geq\dfrac{\min\{\alpha,\beta\}}{2}d^2_{\gra A,\gamma}(x,u^*).
\end{equation}
\end{proposition}
\begin{proof}
Proceeding as in \eqref{e:51}, we observe that 
$(z,(Wx-Wz)/\gamma+u^*)\in\gra A$. It therefore follows from 
Proposition~\ref{p:6}\ref{p:6iv}, \eqref{e:100}, and 
Notation~\ref{n:1} that 
\begin{align}
H_A(x,u^*)&\geq\dfrac{\pair{x-z}{Wx-Wz}}{\gamma}
\nonumber\\
&\geq\dfrac{1}{\gamma}
\brk3{\dfrac{\alpha}{2}\|x-z\|^2+\dfrac{\beta}{2}
\|Wx-Wz\|^2}
\nonumber\\
&\geq\dfrac{\min\{\alpha,\beta\}}{2}
\brk3{\dfrac{\|x-z\|^2}{\gamma}+
\gamma\left\|\dfrac{Wx-Wz}{\gamma}+u^*-u^*\right\|^2}
\nonumber\\
&\geq\dfrac{\min\{\alpha,\beta\}}{2}d^2_{\gra A,\gamma}(x,u^*),
\end{align}
as claimed.
\end{proof}

\begin{remark}
When $\XX$ is a Hilbert space, we deduce 
Proposition~\ref{p:29} from Proposition~\ref{p:30} applied with 
the kernel $W=\Id$, which is $(1,1)$-jointly strongly monotone. 
\end{remark}

\begin{theorem}
\label{t:10}
Let $A\colon\XX\to 2^{\XX^*}$, let $\emp\neq\DD\subset\XX$, let 
$W\colon\DD\to\XX^*$, let $x\in\DD$, let $u^*\in\XX^*$, let
$\gamma\in\RPP$, let $\alpha\in\RPP$, and let $\beta\in\RPP$. 
Suppose that $\DD\cap\dom A\neq\emp$, that one
of properties \ref{p:3i}--\ref{p:3viii} in Proposition~\ref{p:3}
is satisfied, together with one of properties
\ref{t:1i}--\ref{t:1iii} in Theorem~\ref{t:1}, and that $W$ is 
$(\alpha,\beta)$-jointly strongly monotone. Set 
$z=J^W_{\gamma(A-u^*)}x$. Then, 
using Notation~\ref{n:1},  
\begin{equation}
H_A(x,u^*)\geq\dfrac{\pair{x-z}{Wx-Wz}}{\gamma}
\geq\dfrac{\min\{\alpha,\beta\}}{2}d^2_{\gra A,\gamma}(x,u^*).
\end{equation}
\end{theorem}
\begin{proof}
This follows from Theorem~\ref{t:1} and Proposition~\ref{p:30}.
\end{proof}

\section{Lower bounds on the Fenchel--Young function
based on warped proximity operators} 
\label{sec:6} 
Following the pattern adopted in Section~\ref{sec:4}, we derive
from the results of Section~\ref{sec:5} lower bounds on the
Fenchel--Young function of \eqref{e:2}. Comparisons with existing
bounds are made in several examples. To this end, we need the
following extension of \cite[Example~3.1]{Jmaa20} to set-valued
kernels.

\begin{definition}
\label{d:wpx}
Let $\varphi\colon\XX\to\RX$ be proper and let
$K\colon\XX\to 2^{\XX^*}$. Then 
the \emph{warped proximity operator} of $A$ with kernel $K$ is
$\prox_\varphi^K=J_{\partial\varphi}^K=
(K+\partial\varphi)^{-1}\circ K$. 
\end{definition}

Here is a consequence of Lemma~\ref{l:3}, Proposition~\ref{p:6}, 
and Definition~\ref{d:wpx}.

\begin{proposition}
\label{p:9}
Let $\varphi\colon\XX\to\RX$ be proper, let
$W\colon\XX\to 2^{\XX^*}$, let $u^*\in\XX^*$, and let
$\gamma\in\RPP$. Suppose that $x\in\dom W$ and that
$z\in\prox^W_{\gamma(\varphi-u^*)}x$. Then the following hold:
\begin{enumerate}
\item
\label{p:9i}
There exist $x^*\in Wx$ and $z^*\in Wz$ such that 
$L_\varphi(x,u^*)\geq{\pair{x-z}{x^*-z^*}}/{\gamma}$.
\item
\label{p:9ii}
Suppose that $W$ is $\phi$-uniformly monotone. Then 
$L_\varphi(x,u^*)\geq{\phi(\|x-z\|)}/{\gamma}$.
\item
\label{p:9iii}
Suppose that $W$ is $\alpha$-strongly monotone. Then 
$L_\varphi(x,u^*)\geq{\alpha\|x-z\|^2}/{\gamma}$.
\item
\label{p:9iv}
Suppose that $W$ is at most single-valued. Then
$L_\varphi(x,u^*)\geq{\pair{x-z}{Wx-Wz}}/{\gamma}$.
\item
\label{p:9v}
Let $f\in\Gamma_0(\XX)$ be Gateaux differentiable on 
$\dom\nabla f=\intdom f$ and suppose that $W=\nabla f$. Then 
\begin{equation}
\label{e:t6}
L_\varphi(x,u^*)\geq\dfrac{D_f(x,z)+D_f(z,x)}{\gamma}.
\end{equation}
\end{enumerate}
\end{proposition}

\begin{remark}
We can also derive Proposition~\ref{p:9} directly, without invoking
Proposition~\ref{p:6}. Indeed, by Definition~\ref{d:wpx}, 
since $z\in\prox^W_{\gamma(\varphi-u^*)}x$, 
there exists $x^*\in Wx$ such that 
$x^*\in Wz+\partial(\gamma(\varphi-u^*))(z)=
Wz+\gamma\partial\varphi(z)-\gamma u^*$. Hence,
there exists $z^*\in Wz$ such that 
$\gamma^{-1}(x^*-z^*)+u^*\in\partial\varphi(z)$.
In turn, \eqref{e:subdiff} yields
\begin{equation}
(\forall w\in\XX)\quad\pair3{w-z}{\dfrac{x^*-z^*}{\gamma}+u^*}
+\varphi(z)\leq\varphi(w).
\end{equation}
Thus, $\pair{x-z}{x^*-z^*}/\gamma+\pair{x-z}{u^*}
+\varphi(z)\leq\varphi(x)$ and, since $z\in\dom\varphi$, we
conclude that
\begin{equation}
\dfrac{\pair{x-z}{x^*-z^*}}{\gamma}\leq\varphi(x)-\pair{x}{u^*}
+\pair{z}{u^*}-\varphi(z)
\leq\varphi(x)-\pair{x}{u^*}+\varphi^*(u^*)
=L_\varphi(x,u^*).
\end{equation}
This shows \ref{p:9i} in Proposition~\ref{p:9}. Items 
\ref{p:9ii}--\ref{p:9v} in Proposition~\ref{p:9} then follow as in
Proposition~\ref{p:6}.
\end{remark}

Next, we characterize the situation in which $z$ 
in Proposition~\ref{p:9}\ref{p:9iv} is equal to $x$.

\begin{proposition}
\label{p:8}
In the setting of Proposition~\ref{p:9}\ref{p:9iv}, consider the
following statements:
\begin{enumerate}[label={\normalfont[\alph*]}]
\item
\label{p:8a}
$x\in\prox^W_{\gamma(\varphi-u^*)}x$. 
\item
\label{p:8b}
$u^*\in\partial\varphi(x)$.
\item
\label{p:8c}
$L_\varphi(x,u^*)=0$.
\item
\label{p:8d}
$H_{\partial\varphi}(x,u^*)=0$.
\end{enumerate}
Then the following hold:
\begin{enumerate}
\item
\label{p:8i}
\ref{p:8a}$\Leftrightarrow$\ref{p:8b}$\Leftrightarrow$\ref{p:8c}%
$\Rightarrow$\ref{p:8d}.
\item
\label{p:8iii}
Suppose that $\varphi\in\Gamma_0(\XX)$. Then 
\ref{p:8c}$\Leftrightarrow$\ref{p:8d}.
\end{enumerate}
\end{proposition}
\begin{proof}
We recall that $\partial\varphi$ is monotone, and maximally so if
$\varphi\in\Gamma_0(\XX)$. The claims therefore follows from 
\eqref{e:6} and Proposition~\ref{p:7}.
\end{proof}

Let $\varphi\colon\XX\to\RX$ be proper. Setting $A=\partial\varphi$
in Theorem~\ref{t:1} furnishes a first set of conditions under
which $z$ exists in Proposition~\ref{p:9} and is uniquely defined
as $z=\prox^W_{\gamma(\varphi-u^*)}x$. Below, we
refine some of these conditions and add new ones.

\begin{theorem}
\label{t:2}
Let $\varphi\colon\XX\to\RX$ be proper, let 
$\emp\neq\DD\subset\XX$, let $W\colon\DD\to\XX^*$, let $x\in\DD$,
$u^*\in\XX^*$, and let $\gamma\in\RPP$. Suppose that
$\DD\cap\dom\partial\varphi\neq\emp$ and that one of properties
\ref{t:2i}--\ref{t:2iii} below is satisfied, together with one of
properties \ref{t:2iv}--\ref{t:2viii}:
\begin{enumerate}
\item
\label{t:2i}
$W+\gamma\partial\varphi$ is injective.
\item
\label{t:2ii}
$W+\gamma\partial\varphi$ is strictly monotone.
\item
\label{t:2iii}
$W$ is uniformly monotone.
\item
\label{t:2iv}
$Wx+\gamma u^*\in\ran(W+\gamma\partial\varphi)$.
\item
\label{t:2x}
$W$ is monotone, $W+\gamma\partial\varphi$ is maximally monotone,
$\ran W+\gamma\ran\partial\varphi=\XX^*$, and one of the following
is satisfied:
\begin{enumerate}
\item
\label{t:2xa}
$W$ is $3^*$ monotone. 
\item
\label{t:2xc}
$\dom W\subset\dom\partial\varphi$.
\end{enumerate}
\item
\label{t:2xi}
$W=\partial f$, with $f\in\Gamma_0(\XX)$,
$\varphi\in\Gamma_0(\XX)$, 
the cone generated by $\dom f-\dom\varphi$ is a closed vector
subspace of $\XX$, and
$\dom\partial f^*+\gamma\dom\partial\varphi^*=\XX^*$.
\item
\label{t:2viii}
$W$ is monotone, $W+\gamma\partial\varphi$ is maximally monotone, 
and one of the following is satisfied:
\begin{enumerate}
\item
\label{t:2viiia}
$W$ is $3^*$ monotone and surjective.
\item
\label{t:2viiib}
$\partial\varphi$ is surjective.
\end{enumerate}
\end{enumerate}
Set $z=\prox^W_{\gamma(\varphi-u^*)}x$. Then 
$L_\varphi(x,u^*)\geq{\pair{x-z}{Wx-Wz}}/{\gamma}$.
\end{theorem}
\begin{proof}
We apply Theorem~\ref{t:1} with $A=\partial\varphi$, taking 
into account the fact that, by \cite[Proposition~32.42]{Zei90B}, 
\begin{equation}
\label{e:7}
A\;\text{is $3^*$ monotone}.
\end{equation}

\ref{t:2i}: Theorem~\ref{t:1}\ref{t:1i}.

\ref{t:2ii}: Theorem~\ref{t:1}\ref{t:1ii}.

\ref{t:2iii}: Theorem~\ref{t:1}\ref{t:1iii} and \eqref{e:7}.

\ref{t:2iv}: Proposition~\ref{p:3}\ref{p:3i}.

\ref{t:2x}: Proposition~\ref{p:3}\ref{p:3vii} and \eqref{e:7}.

\ref{t:2xi}: This follows from Proposition~\ref{p:3}\ref{p:3ix}
and \eqref{e:7}. Indeed, by the Attouch--Br\'ezis theorem
\cite[Theorem~18.2]{Simo08}, 
$W+\gamma A=\partial f+\gamma\partial\varphi=
\partial(f+\gamma\varphi)$. However, since 
$f+\gamma\varphi\in\Gamma_0(\XX)$, Rockafellar's theorem
\cite[Theorem~18.7]{Simo08} asserts that this operator 
is maximally monotone.

\ref{t:2viii}: Proposition~\ref{p:3}\ref{p:3viii} and 
\eqref{e:7}.
\end{proof}

We now focus on the case when $W$ is the gradient of a Legendre
function $f$. The warped proximity operator 
$\prox^{\nabla f}_{\varphi}$ becomes the Bregman proximity operator
of $\varphi$ studied in \cite[Section~3.4]{Sico03}.

\begin{proposition}
\label{p:5}
Let $\varphi\in\Gamma_0(\XX)$, let $f\in\Gamma_0(\XX)$ be a
Legendre function, let $x\in\intdom f$, let $u^*\in\XX^*$, and 
let $\gamma\in\RPP$. Suppose that 
$(\intdom f)\cap\dom\partial\varphi\neq\emp$ and that one of the
following holds:
\begin{enumerate}
\item
\label{p:5i}
$\gamma u^*+\intdom f^*\subset
\ran(\nabla f+\gamma\partial\varphi)$.
\item
\label{p:5ii}
$(\intdom f^*)+\gamma\dom\partial\varphi^*=\XX^*$.
\item
\label{p:5iii}
$f+\gamma\varphi$ is cofinite.
\item
\label{p:5iv}
$f+\gamma\varphi$ is supercoercive.
\item
\label{p:5v}
$\dom f\cap\dom\varphi$ is bounded.
\end{enumerate}
Set $z=\prox^{\nabla f}_{\gamma(\varphi-u^*)}x$. Then
\begin{equation}
\label{e:pll}
L_\varphi(x,u^*)\geq\dfrac{D_f(x,z)+D_f(z,x)}{\gamma}.
\end{equation}
\end{proposition}
\begin{proof}
\ref{p:5i}--\ref{p:5ii}: These follow from Proposition~\ref{p:2},
\eqref{e:7}, and Lemma~\ref{l:3}.

\ref{p:5iii}$\Rightarrow$\ref{p:5i}: We have 
$f+\gamma\varphi\in\Gamma_0(\XX)$, hence 
$(f+\gamma\varphi)^*\in\Gamma_0(\XX^*)$. Therefore, by
\cite[Corollary~47.7, Theorem~47.A(ii), Theorem~51.A(ii), and %
Theorem~47.B]{Zeid85} and \cite[Theorem~5.6]{Ccma01},
\begin{align}
\XX^*&=\intdom(f+\gamma\varphi)^*\nonumber\\
&\subset\dom\partial(f+\gamma\varphi)^*\nonumber\\
&=\ran\brk1{\partial(f+\gamma\varphi)^*}^{-1}\nonumber\\
&=\ran\partial(f+\gamma\varphi)\nonumber\\
&=\ran\brk{\partial f+\gamma\partial\varphi}\nonumber\\
&=\ran\brk{\nabla f+\gamma\partial\varphi},
\end{align}
which confirms that 
$\ran(\nabla f+\gamma\partial\varphi)=\XX^*$.

\ref{p:5iv}$\Rightarrow$\ref{p:5iii}: \cite[Theorem~3.4]{Ccma01}.

\ref{p:5v}$\Rightarrow$\ref{p:5iv}: Clear.
\end{proof}

\begin{example}
\label{ex:81}
When $\XX$ is a Hilbert space, we recover
Proposition~\ref{p:21}\ref{p:21ii} as a special case of
Proposition~\ref{p:5}\ref{p:5ii} with $f=\|\cdot\|^2/2$. It is also
a special case of Proposition~\ref{p:9}\ref{p:9iv} with $W=\Id$.
\end{example}

\begin{example}
\label{ex:11}
Let $\varphi\in\Gamma_0(\XX)$ be a Legendre function, 
let $x\in\intdom\varphi$, let $u^*\in\intdom\varphi^*$, and 
let $\gamma\in\RPP$. Then 
\begin{equation}
\label{e:11}
L_\varphi(x,u^*)\geq
\dfrac{\pair3{x-\nabla\varphi^*\brk2{(1+\gamma)^{-1}
\brk1{\nabla\varphi(x)+\gamma u^*}}}{\nabla\varphi(x)-u^*}}
{1+\gamma}.
\end{equation}
\end{example}
\begin{proof}
It follows from the results of \cite[Section~5]{Ccma01} that 
$\dom\partial\varphi^*=\intdom\varphi^*$ and 
$\nabla\varphi\colon\intdom\varphi\to\intdom\varphi^*$ is a
bijection with inverse $\nabla\varphi^*$. We establish the claim 
by setting $f=\varphi$ in Proposition~\ref{p:5}\ref{p:5i}. We 
first observe that
\begin{equation}
\gamma u^*+\intdom\varphi^*
\subset(\gamma+1)\intdom\varphi^*
=(1+\gamma)\dom\partial\varphi^*
=(1+\gamma)\ran\partial\varphi
=\ran\brk1{(1+\gamma)\partial\varphi}.
\end{equation}
As in Proposition~\ref{p:5}\ref{p:5i}, set 
$z=\prox^{\nabla\varphi}_{\gamma(\varphi-u^*)}x=
((1+\gamma)\nabla\varphi-\gamma u^*)^{-1}(\nabla\varphi(x))$.
Then $\nabla\varphi(x)+\gamma u^*=(1+\gamma)\nabla\varphi(z)$ and
therefore 
$z=\nabla\varphi^*((1+\gamma)^{-1}(\nabla\varphi(x)+\gamma u^*))$.
Thus, 
\begin{align}
\dfrac{\pair1{x-z}{\nabla\varphi(x)-\nabla\varphi(z)}}{\gamma}
&=\dfrac{\pair2{x-z}{\nabla\varphi(x)-(1+\gamma)^{-1}\brk1{\nabla 
\varphi(x)+\gamma u^*}}}{\gamma}\nonumber\\
&=\dfrac{\pair2{x-z}{\gamma\brk1{\nabla\varphi(x)-u^*}}}{
\gamma(1+\gamma)}\nonumber\\
&=\dfrac{\pair3{x-\nabla\varphi^*\brk2{(1+\gamma)^{-1}
\brk1{\nabla\varphi(x)+\gamma u^*}}}{\nabla\varphi(x)-u^*}}
{1+\gamma}
\end{align}
which, in view of \eqref{e:pll} and \eqref{e:B}, yields
\eqref{e:11}.
\end{proof}

\begin{example}
\label{ex:15}
Suppose that $\XX$ is a Hilbert space with scalar product
$\scal{\cdot}{\cdot}$, 
let $\psi\in\Gamma_0(\XX)$ be a Legendre function, 
let $\widetilde{\psi}\colon\XX\to\RR\colon x\mapsto\inf_{y\in\XX}
(\psi(y)+\|x-y\|^2/2)$ be its Moreau envelope, and
set $\varphi=\|\cdot\|^2/2+\psi$. 
Let $x\in\inte\dom\psi$, $u^*\in\XX$, and $\gamma\in\RPP$. 
Then
\begin{equation}
L_\varphi(x,u^*)=\dfrac{1}{2}\|x-u^*\|^2+\psi(x)-
\widetilde{\psi}(u^*).
\end{equation}
Further, Proposition~\ref{p:21}\ref{p:21ii} gives 
\begin{equation}
\label{e:8}
L_\varphi(x,u^*)\geq
\dfrac{\left\|x-\prox_{\gamma(1+\gamma)^{-1}\psi}
\brk3{\dfrac{x+\gamma u^*}{1+\gamma}}\right\|^2}{\gamma},
\end{equation}
while Example~\ref{ex:11} gives
\begin{equation}
\label{e:91}
L_\varphi(x,u^*)\geq\dfrac{\scal{x-z}{x+\nabla\psi(x)-u^*}}
{1+\gamma},
\quad\text{where}\quad
z=\prox_\psi\brk3{\dfrac{x+\nabla\psi(x)+\gamma u^*}
{1+\gamma}}. 
\end{equation}
\end{example}
\begin{proof}
By \cite[Example~13.4]{Livre1}, 
$\varphi^*=\|\cdot\|^2/2-\widetilde{\psi}$. Therefore, 
\begin{equation}
\label{e:f8}
L_\varphi(x,u^*)=\dfrac{1}{2}\|u^*\|^2+\dfrac{1}{2}\|x\|^2
-\scal{x}{u^*}+\psi(x)-\widetilde{\psi}(u^*)
=\dfrac{1}{2}\|x-u^*\|^2+\psi(x)-\widetilde{\psi}(u^*).
\end{equation}
We derive \eqref{e:8} from Proposition~\ref{p:21}\ref{p:21ii} 
and \cite[Proposition~24.8(i)]{Livre1}. We also note that 
$\varphi$ is a Legendre function. Moreover, $\dom\varphi^*=\XX$ and 
$\nabla\varphi^*=\prox_\psi$ \cite[Proposition~12.30]{Livre1}. 
In turn, we derive \eqref{e:91} from Example~\ref{ex:11}. 
\end{proof}

\begin{remark}
\label{r:410}
In Example~\ref{ex:15}, set $\psi=\|\cdot\|^2/2$ and 
$\gamma=1$. Then $L_\varphi(x,u^*)={\|2x-u^*\|^2}/{4}$.
Furthermore, Proposition~\ref{p:21}\ref{p:21ii} yields 
\begin{equation}
L_\varphi(x,u^*)\geq
\left\|x-\prox_{\|\cdot\|^2/4}\brk3{\dfrac{x+u^*}{2}}\right\|^2
=\dfrac{\|2x-u^*\|^2}{9}.
\end{equation}
On the other hand, since $\nabla\psi=\Id$ in \eqref{e:91}, we have
\begin{equation}
z=\prox_\psi\brk3{\dfrac{2x+u^*}{2}}
=\dfrac{2x+u^*}{4}
\end{equation}
and therefore the minorization of Example~\ref{ex:11} becomes
\begin{equation}
L_\varphi(x,u^*)\geq\dfrac{\scal{x-z}{2x-u^*}}{2}
=\dfrac{\|2x-u^*\|^2}{8}.
\end{equation}
The lower bound produced by Example~\ref{ex:11} is therefore 
sharper than that of Proposition~\ref{p:21}\ref{p:21ii}.
\end{remark}

Our lower bounds on the Haraux and Fenchel--Young
functions are new, even in Euclidean spaces. Here are two examples
in which they are compared to the lower bound of
Proposition~\ref{p:21}\ref{p:21ii}. 

\begin{example}
\label{ex:3}
Let $\XX$ be the standard Euclidean space $\RR^N$ and 
$I=\{1,\dots,N\}$. Consider the negative Burg entropy function
\begin{equation}
\varphi\colon\XX\to\RX\colon x=(\xi_i)_{i\in I}\mapsto
\begin{cases}
-\displaystyle\sum_{i\in I}\ln(\xi_i),&\text{if}\;\;
x\in\RPP^N;\\
\pinf,&\text{otherwise.}
\end{cases}
\end{equation}
Let $x\in\RPP^N$, $u^*=(\mu^*_i)_{i\in I}\in\RMM^N$, and
$\gamma\in\RPP$. Then 
\begin{equation}
L_\varphi(x,u^*)=
-N-\sum_{i\in I}\brk1{\ln(-\xi_i\mu^*_i)+\xi_i\mu^*_i}
\end{equation}
and Example~\ref{ex:11} gives
\begin{equation}
L_\varphi(x,u^*)\geq\sum_{i\in I}
\dfrac{\gamma|1+\xi_i\mu^*_i|^2}{(1+\gamma)(1-\gamma\xi_i\mu^*_i)}.
\end{equation}
Let us observe that, if $N=1$, this lower bound becomes
${\gamma|1+\xi_1\mu^*_1|^2}/((1+\gamma)(1-\gamma\xi_1\mu^*_1))$.
In comparison, we derive from \cite[Example~24.40]{Livre1} 
that the lower bound given by
Proposition~\ref{p:21}\ref{p:21ii} is 
\begin{equation}
\dfrac{\left|\xi_1-\gamma\mu^*_1-\sqrt{|\xi_1+\gamma\mu^*_1|^2
+4\gamma}\right|^2}{4\gamma}.
\end{equation}
We graph these two bounds in Figure~\ref{f:1} for different 
values of $\gamma$, which shows that the bound provided by 
Proposition~\ref{p:21}\ref{p:21ii} is not the best.
\end{example}
\begin{proof}
By \cite[Example~13.2(iii) and Proposition~13.30]{Livre1}, 
$\varphi^*(u^*)=-N-\sum_{i\in I}\ln(-\mu^*_i)$. Thus,
\begin{equation}
L_\varphi(x,u^*)=
-N-\sum_{i\in I}\brk1{\ln(-\xi_i\mu^*_i)+\xi_i\mu^*_i}.
\end{equation}
Remark that $\varphi$ is a Legendre function with
$\dom\nabla\varphi=\inte\dom\varphi=\RPP^N$ and
\begin{equation}
\brk1{\forall y=(\eta_i)_{i\in I}\in\RPP^N}\quad\nabla
\varphi(y)=\brk1{-{1}/{\eta_i}}_{i\in I}.
\end{equation}
Further, it is clear that $u^*\in\RMM^N=\inte\dom\varphi^*$.
By Example~\ref{ex:11}, 
\begin{equation}
z=(\zeta_i)_{i\in I}=
\nabla\varphi^*\brk2{(1+\gamma)^{-1}\brk1{\nabla\varphi(x)+
\gamma u^*}}
\end{equation}
is well defined and 
$\nabla\varphi(x)+\gamma u^*=(1+\gamma)\nabla\varphi(z)$, which
yields $(\forall i\in I)$ 
$(1-\gamma\xi_i\mu^*_i)/\xi_i=(1+\gamma)/\zeta_i$.
Hence,
\begin{equation}
z=\brk3{\dfrac{(1+\gamma)\xi_i}{1-\gamma\xi_i
\mu^*_i}}_{i\in I}\in\RPP^N.
\end{equation}
We thus derive from \eqref{e:B} that
\begin{equation}
D_\varphi(x,z)+D_\varphi(z,x)
=\sum_{i\in I}\brk3{-2+\dfrac{\xi_i}{\zeta_i}+
\dfrac{\zeta_i}{\xi_i}}
=\sum_{i\in I}
\dfrac{\gamma|1+\xi_i\mu^*_i|^2}{(1+\gamma)(1-\gamma\xi_i\mu^*_i)},
\end{equation}
which concludes the proof.
\end{proof}

\begin{figure}[h]
\centering
\begin{subfigure}{0.45\textwidth}
\centering
\includegraphics[width=\linewidth]{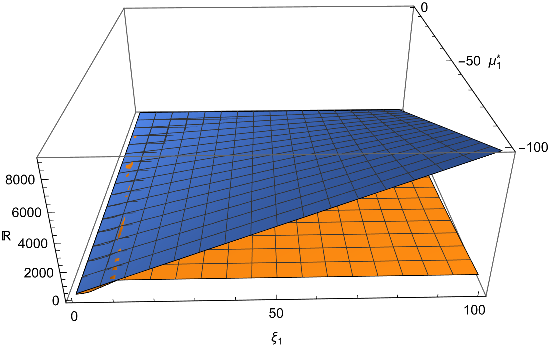}
\caption{Example~\ref{ex:3}, $\gamma=0.1$.}
\label{f:1a}
\end{subfigure}
\begin{subfigure}{0.45\textwidth}
\centering
~~~~~~\includegraphics[width=\linewidth]{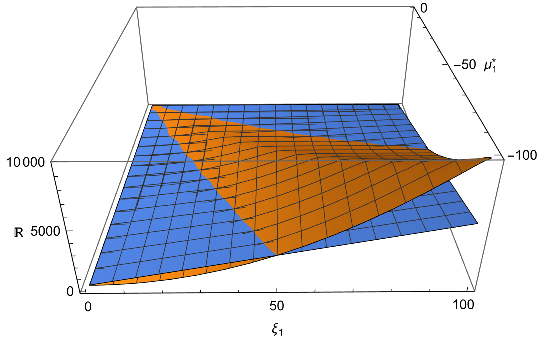}
\caption{Example~\ref{ex:3}, $\gamma=1$.}
\label{f:1b}
\end{subfigure}
\vspace{4mm}
\begin{subfigure}{0.45\textwidth}
\vspace{11mm}
\centering
\includegraphics[width=\linewidth]{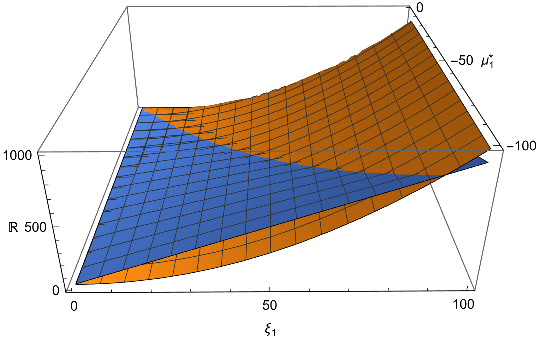}
\caption{Example~\ref{ex:3}, $\gamma=10$.}
\label{f:1c}
\end{subfigure}
\begin{subfigure}{0.45\textwidth}
\vspace{9mm}
\centering
~~~~~~\includegraphics[width=\linewidth]{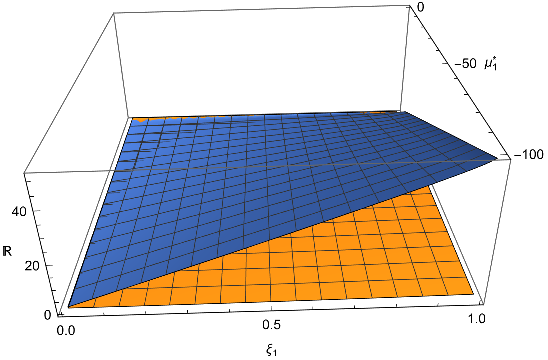}
\caption{Example~\ref{ex:4}, $\gamma=1$.}
\label{f:1d}
\end{subfigure}
\caption{Lower bounds of Examples~\ref{ex:3} and \ref{ex:4}. 
In blue, the lower bound of Proposition~\ref{p:5} and, in orange,
the lower bound of Proposition~\ref{p:21}\ref{p:21ii}.}
\label{f:1}
\end{figure}

\begin{example}
\label{ex:4}
Let $\XX$ be the standard Euclidean space $\RR^N$ and 
$I=\{1,\dots,N\}$. Set
\begin{multline}
\brk1{\forall x=(\xi_i)_{i\in I}\in\XX}\quad
I_1(x)=\menge{i\in I}{\xi_i\in\RP},\;\; 
I_2(x)=\menge{i\in I}{\xi_i\in\RPP},\\
I_3(x)=\menge{i\in I}{\xi_i\in[0,1]},\;\;\text{and}\;\;
I_4(x)=\menge{i\in I}{\xi_i\in\zeroun}.
\end{multline}
We consider the negative Boltzmann--Shannon entropy function
\begin{equation}
\varphi\colon\XX\to\RX\colon x\mapsto
\begin{cases}
\displaystyle\sum_{i\in I_2(x)}
\brk1{\xi_i\ln(\xi_i)-\xi_i},&\text{if}\;\;
I=I_1(x)\;\;\text{and}\;\;I_2(x)\neq\emp;\\
0,&\text{if}\;\;I=I_1(x)\smallsetminus I_2(x);\\
\pinf,&\text{otherwise,}
\end{cases}
\end{equation}
and let $f$ be the Fermi--Dirac entropy function
\begin{equation}
f\colon\XX\to\RX\colon x\mapsto
\begin{cases}
\displaystyle\sum_{i\in I_4(x)}
\brk1{\xi_i\ln(\xi_i)+(1-\xi_i)\ln(1-\xi_i)}
,&\text{if}\;\;I=I_3(x)\;\;\text{and}\;\;I_4(x)\neq\emp;\\
0,&\text{if}\;\;I=I_3(x)\smallsetminus I_4(x);\\
\pinf,&\text{otherwise.}
\end{cases}
\end{equation}
Let $x\in\zeroun^N$, $u^*=(\mu^*_i)_{i\in I}\in\XX$, and
$\gamma\in\RPP$. Set 
\begin{equation}
(\forall i\in I)\quad 
\zeta_i=-\dfrac{\xi_ie^{\gamma\mu^*_i}}{2(1-\xi_i)}
+\sqrt{\dfrac{\xi^2_ie^{2\gamma\mu^*_i}}{4|1-\xi_i|^2}
+\dfrac{\xi_ie^{\gamma\mu^*_i}}{1-\xi_i}}.
\end{equation}
Then 
\begin{equation}
L_\varphi(x,u^*)=
\sum_{i\in I}\brk1{\xi_i\ln(\xi_i)-\xi_i+e^{\mu^*_i}-\xi_i\mu^*_i}
\end{equation}
and Proposition~\ref{p:5}\ref{p:5v} gives
\begin{equation}
L_\varphi(x,u^*)
\geq\dfrac{1}{\gamma}\sum_{i\in I}(\xi_i-\zeta_i)
\ln\brk3{\dfrac{\xi_i(1-\zeta_i)}{\zeta_i(1-\xi_i)}}.
\end{equation}
Thus, if $N=1$ and $\gamma=1$, the above lower
bound is
\begin{equation}
\label{e:93}
(\xi_1-\zeta_1)
\ln\brk3{\dfrac{\xi_1(1-\zeta_1)}{\zeta_1(1-\xi_1)}}.
\end{equation}
By contrast, since 
$\prox_\varphi\colon\xi_1\mapsto\Lambda(e^{\xi_1})$, 
where $\Lambda$ is the Lambert W-function
\cite[Example~24.39]{Livre1}, the lower bound of 
Proposition~\ref{p:21}\ref{p:21ii} is
\begin{align}
\bigl|\xi_1-\Lambda\brk1{e^{\xi_1+\mu^*_1}}\bigr|^2.
\end{align}
We illustrate these bounds in Figure~\ref{f:1}, which 
shows that the bound given by Proposition~\ref{p:21}\ref{p:21ii} is
not the best. The bound \eqref{e:93} provided by
Proposition~\ref{p:5} is also easier to compute.
\end{example}
\begin{proof}
Using \cite[Example~13.2(v) and Proposition~13.30]{Livre1}, 
we obtain $\varphi^*(u^*)=\sum_{i\in I}e^{\mu^*_i}$ and hence
\begin{equation}
L_\varphi(x,u^*)=
\sum_{i\in I}\brk1{\xi_i\ln(\xi_i)-\xi_i+e^{\mu^*_i}
-\xi_i\mu^*_i}.
\end{equation}
Further, $f$ is a Legendre function 
\cite[Sections~5 and 6]{Baus97}, with
$\dom\nabla f=\inte\dom f=\zeroun^N$, where 
\begin{equation}
\brk2{\forall y=(\eta_i)_{i\in I}\in\zeroun^N}\quad\nabla
f(y)=\brk1{\ln(\eta_i)-\ln(1-\eta_i)}_{i\in I}.
\end{equation}
Additionally,
$(\inte\dom f)\cap\dom\partial\varphi=\zeroun^N\neq\emp$. Thus, 
by Proposition~\ref{p:5}\ref{p:5v} and \cite[Example~4.4]{Joca16}, 
$z=(\zeta_i)_{i\in I}$ exists and
\begin{align}
(\forall i\in I)\quad
\zeta_i&={-\dfrac{e^{\ln(\xi_i)-\ln(1-\xi_i)+\gamma\mu^*_i}}{2}
+\sqrt{\dfrac{e^{2\brk1{\ln(\xi_i)-\ln(1-\xi_i)+\gamma\mu^*_i}}}
{4}+e^{\ln(\xi_i)-\ln(1-\xi_i)+\gamma\mu^*_i}}}
\nonumber\\
&={-\dfrac{\xi_ie^{\gamma\mu^*_i}}{2(1-\xi_i)}
+\sqrt{\dfrac{\xi_i^2}{4|1-\xi_i|^2}e^{2\gamma\mu^*_i}
+\dfrac{\xi_i}{1-\xi_i}e^{\gamma\mu^*_i}}}.
\end{align}
Hence, 
\begin{equation}
D_f(x,z)+D_f(z,x)
=\sum_{i\in I}(\xi_i-\zeta_i)
\ln\brk3{\dfrac{\xi_i(1-\zeta_i)}{\zeta_i(1-\xi_i)}}
\end{equation}
and the conclusion follows.
\end{proof}

\section{Applications to monotone inclusions} 
\label{sec:7} 
As discussed in Section~\ref{sec:1}, finding lower bounds on the
Haraux function --- and hence on the Fenchel--Young function via
Lemma~\ref{l:3} --- is of both theoretical and practical interest.
Thus, applications to inverse problems in are discussed in
\cite{Andr25}, applications to the asymptotic properties of
families of set-valued operators in \cite{Atto18}, \cite{Peno06},
and Proposition~\ref{p:25}, applications to machine learning in
\cite{Blon20,Rako24}, applications to the strong Fitzpatrick
inequality in \cite{Bura25}, applications to convex analysis and
convex programming in \cite{Car23a}, and applications to optimal
transportation in \cite{Car23a} and \cite{Car23b}. In this section,
we propose to apply the bounds of Sections~\ref{sec:5} and
\ref{sec:6} to the area of composite monotone inclusion problems. 

Let $A\colon\XX\to 2^{\XX^*}$ be a maximally monotone operator. 
The Haraux function $H_A$ of \eqref{e:1} is defined on
the primal-dual space $\XX\times\XX^*$. We can employ it to induce
a primal-primal function on $\XX\times\XX$. To this end, let 
$B\colon\XX\to 2^{\XX^*}$ be maximally monotone and let 
$S\colon\dom B\to\XX^*\colon x\mapsto Sx\in Bx$ be a selection of
$B$. We associate with $H_A$ the function
\begin{equation}
\label{e:17}
G_{A,S}\colon\XX\times\XX\to\RPX\colon(x,y)\mapsto H_A(x,-Sy).
\end{equation}
This function can be used as a penalty function to detect whether a
point $(x,y)\in\XX\times\XX$ is a zero of the direct sum operator 
$A\oplus B$ since Proposition~\ref{p:7} yields
\begin{equation}
\label{e:18}
G_{A,S}(x,y)=0\quad\Leftrightarrow\quad -Sy\in Ax\quad
\Leftrightarrow\quad(x,y)\in\zer(A\oplus B).
\end{equation}
As $H_A$, and therefore $G_{A,S}$, can be hard to evaluate, our
results provide more tractable lower bounds to test the violation
of the constraint $(x,y)\in\zer(A\oplus B)$. We can further
specialize this construction to address the inclusion
$0\in Ax+Bx$, a generic model which covers a wide range of
applications; see \cite{Acnu24} and its bibliography. 
We introduce the primal function
\begin{equation}
\vartheta_{A,S}\colon\XX\to\RPX\colon x\mapsto G_{A,S}(x,x)
=H_A(x,-Sx)
\end{equation}
to gauge the membership of a point $x\in\XX$ in $\zer(A+B)$. 
Indeed, \eqref{e:18} yields
\begin{equation}
\label{e:74}
\vartheta_{A,S}(x)=0\quad\Leftrightarrow\quad -Sx\in Ax\quad
\Leftrightarrow\quad x\in\zer(A+B).
\end{equation}
To illustrate our lower bounds in this context, it is assumed
henceforth that $B\colon\XX\to\XX^*$ is single-valued. 

In the setting of Theorem~\ref{t:1}, let $x\in\DD$, set 
$u^*=-Bx$, and consider the kernel $K=W-\gamma B$. Then 
Proposition~\ref{p:12} yields
\begin{align}
\label{e:alain}
z
&=J_{\gamma(A-u^*)}^Wx\nonumber\\
&=(W+\gamma A)^{-1}(Wx+\gamma u^*)\nonumber\\
&=\brk1{W-\gamma B+\gamma (A+B)}^{-1}(Wx-\gamma Bx)\nonumber\\
&=J_{\gamma(A+B)}^Kx,
\end{align}
and the lower bound of Theorem~\ref{t:1} is therefore
\begin{equation}
\label{e:g1}
\vartheta_{A,B}(x)=H_A(x,-Bx)\geq
\dfrac{\pair1{x-J_{\gamma(A+B)}^Kx}
{Wx-W\brk1{J_{\gamma(A+B)}^Kx}}}{\gamma}.
\end{equation}

\begin{example}
\label{ex:g1}
In the setting of Example~\ref{ex:24}, \eqref{e:g1} yields
\begin{equation}
\label{e:g4}
\vartheta_{A,B}(x)
\geq\dfrac{\phi\brk2{\bigl\|x-J_{\gamma(A+B)}^Kx\bigr\|}}{\gamma},
\quad\text{where}\quad K=W-\gamma B.
\end{equation}
\end{example}

\begin{example}
\label{ex:g2}
In the setting of Proposition~\ref{p:2}, \eqref{e:g1} yields
\begin{equation}
\label{e:g5}
\vartheta_{A,B}(x)\geq\dfrac{D_f\brk2{x,J_{\gamma(A+B)}^Kx}+D_f
\brk2{J_{\gamma(A+B)}^Kx,x}}{\gamma},
\quad\text{where}\quad K=\nabla f-\gamma B.
\end{equation}
In this case,
$J_{\gamma(A+B)}^K=(\nabla f+\gamma A)^{-1}\circ
(\nabla f-\gamma B)$ is the Bregman forward-backward splitting
operator studied in \cite{Svva21}. 
\end{example}

\begin{example}
\label{ex:g3}
If $\XX$ is Hilbertian and $f=\|\cdot\|^2/2$ in
Example~\ref{ex:g2}, then \eqref{e:g5} becomes
\begin{equation}
\label{e:g7}
\vartheta_{A,B}(x)\geq
\dfrac{\bigl\|x-J_{\gamma A}(x-\gamma Bx)\bigr\|^2}{\gamma},
\end{equation}
which effectively splits $A$ and $B$ and is computable in terms of 
the standard resolvent $J_{\gamma A}$. 
\end{example}

\begin{example}
\label{ex:pd}
We consider a primal-dual composite problem discussed in
\cite{Joca16}. Let $\YY\neq\{0\}$ be a reflexive real Banach space,
let $\EEE$ be the standard product vector space $\XX\times\YY^*$ 
equipped with the norm $(x,y^*)\mapsto\sqrt{\|x\|^2+\|y^*\|^2}$, 
and let $\EEE^*$ be its topological dual, that is, 
$\XX^*\times\YY$ equipped with the norm 
$(x^*,y)\mapsto\sqrt{\|x^*\|^2+\|y\|^2}$.
Let $C\colon\XX\to 2^{\XX^*}$ and $D\colon\YY\to 2^{\YY^*}$ be
maximally monotone, and let $L\colon\XX\to\YY$ be linear and
bounded. Under consideration is the primal-dual system
\begin{equation}
\label{e:pd}
\text{find}\;\:(x,y^*)\in\EEE\;\:\text{such that}\;\:
\begin{cases}
0\in Cx+L^*\brk1{D(Lx)}\\
0\in-L\brk1{C^{-1}(-L^*y^*)}+D^{-1}y^*.
\end{cases}
\end{equation}
Let us introduce the operators
\begin{equation}
\label{e:31z}
\begin{cases}
\boldsymbol{A}\colon\EEE\to 2^{\EEE^*}\colon(x,y^*)\mapsto
Cx\,\times\,D^{-1}y^*\\
\boldsymbol{B}\colon\EEE\to\EEE^*\colon(x,y^*)\mapsto(L^*y^*,-Lx).
\end{cases}
\end{equation}
As shown in \cite[Section~2.1]{Joca16}, $\boldsymbol{A}$ and
$\boldsymbol{B}$ are maximally monotone and every point in 
the Kuhn--Tucker set $\zer(\boldsymbol{A}+\boldsymbol{B})$ solves 
\eqref{e:pd}. Let $\gamma\in\RPP$, $\DD_{\XX}\subset\XX$, and
$\DD_{\YY^*}\subset\YY^*$. Let
$W_{\XX}\colon\DD_{\XX}\to\XX^*$ and
$W_{\YY^*}\colon\DD_{\YY^*}\to\YY$ be such that
$W_{\XX}+\gamma C$ and 
$W_{\YY^*}+\gamma D^{-1}$ are surjective and injective. Further,
set $\DDD=\DD_{\XX}\times\DD_{\YY^*}$ and
$\boldsymbol{W}\colon\DDD\to\EEE^*\colon(x,y^*)
\mapsto(W_{\XX}x,W_{\YY^*}y^*)$. Then
$\boldsymbol{W}+\gamma\boldsymbol{A}$ is surjective and injective,
which confirms that properties \ref{p:3ii} of Proposition~\ref{p:3}
and \ref{t:1i} of Theorem~\ref{t:1} are satisfied. 
Additionally, we define
$\boldsymbol{K}=\boldsymbol{W}-\gamma\boldsymbol{B}
\colon\DDD\to\EEE^*\colon(x,y^*)
\mapsto(W_{\XX}x-\gamma L^*y^*,W_{\YY^*}y^*+\gamma Lx)$.
Now, let us fix $\boldsymbol{x}=(x,y^*)\in\DDD$. Then, as 
in \eqref{e:74}, how close $\boldsymbol{x}$ is to being a
Kuhn--Tucker point can be gauged by the value of the penalty
function $\vartheta(\boldsymbol{x})=
\vartheta_{\boldsymbol{A},\boldsymbol{B}}(\boldsymbol{x})$. 
Note that \eqref{e:alain} and Proposition~\ref{p:12} entail that 
\begin{align}
J_{\gamma(\boldsymbol{A}+\boldsymbol{B})}^{\boldsymbol{K}}
\boldsymbol{x}
&=\brk1{\boldsymbol{K}+\gamma(\boldsymbol{A}+\boldsymbol{B})}^{-1}
(\boldsymbol{K}\boldsymbol{x})\nonumber\\
&=\brk2{(W_{\XX}+\gamma C)^{-1}
(W_{\XX}x-\gamma L^*y^*),(W_{\YY^*}+\gamma D^{-1})^{-1}
(W_{\YY^*}y^*+\gamma Lx)}\nonumber\\
&=\brk2{J_{\gamma(C+L^*y^*)}^{W_{\XX}}x,
J_{\gamma(D^{-1}-Lx)}^{W_{\YY^*}}y^*}
\end{align}
is well defined. Therefore, \eqref{e:g1} applied to 
$(\boldsymbol{A},\boldsymbol{B},\boldsymbol{K})$ in $\EEE$ yields
\begin{align}
\label{e:fin}
\vartheta(\boldsymbol{x})
&\geq\dfrac{1}{\gamma}\biggl(
\pair2{x-J_{\gamma(C+L^*y^*)}^{W_{\XX}}x}
{W_{\XX}x-W_{\XX}\brk2{J_{\gamma(C+L^*y^*)}^{W_{\XX}}x}}
\nonumber\\
&\qquad\;\;+\pair2{y^*-J_{\gamma(D^{-1}-Lx)}^{W_{\YY^*}}y^*}
{W_{\YY^*}y^*-W_{\YY^*}\brk2{J_{\gamma(D^{-1}-Lx)}^{W_{\YY^*}}y^*}}
\biggr).
\end{align}
For example, suppose that $W_{\XX}$ is maximally monotone, 
that the cone generated by $\DD_{\XX}-\dom C$ is a closed vector 
subspace of $\XX$, and that $W_{\XX}$ is $\phi_{\XX}$-uniformly 
monotone with $\phi_{\XX}(t)/t\to\pinf$ as $t\to\pinf$. 
Likewise, suppose that $W_{\YY^*}$ is maximally monotone, that
the cone generated by $\DD_{\YY^*}-\ran D$ is a closed vector 
subspace of $\YY^*$, and that $W_{\YY^*}$ is
$\phi_{\YY^*}$-uniformly monotone with $\phi_{\YY^*}(t)/t\to\pinf$
as $t\to\pinf$. Then we deduce from \eqref{e:fin} and
Example~\ref{ex:24} applied in $\XX$ and in $\YY^*$, and from
Proposition~\ref{p:12} that
\begin{align}
\label{e:g9}
\vartheta(\boldsymbol{x})
&\geq\dfrac{\phi_{\XX}\brk2{\bigl\|
x-J_{\gamma(C+L^*y^*)}^{W_{\XX}}x\bigr\|}
+\phi_{\YY^*}\brk2{\bigl\|
y^*-J_{\gamma(D^{-1}-Lx)}^{W_{\YY^*}}y^*\bigr\|}}
{\gamma}\nonumber\\
&=\dfrac{\phi_{\XX}\brk2{
\bigl\|x-(W_{\XX}+\gamma C)^{-1}(W_{\XX}x-\gamma L^*y^*)\bigr\|}
+\phi_{\YY^*}\brk2{\bigl\|y^*-(W_{\YY^*}+\gamma D^{-1})^{-1}
(W_{\YY^*}y^*+\gamma Lx)\bigr\|}}{\gamma}.
\end{align}
In particular, if $\XX$ and $\YY$ are Hilbertian, 
$W_{\XX}=\Id_{\XX}$, and $W_{\YY^*}=\Id_{\YY^*}$, then all the 
above hypotheses are satisfied with 
$\phi_{\XX}=\phi_{\YY^*}=|\cdot|^2$ and we obtain, for 
$\boldsymbol{x}=(x,y^*)\in\XX\times\YY$,
\begin{equation}
\label{e:fz}
\vartheta(\boldsymbol{x})
\geq\dfrac{
\bigl\|x-J_{\gamma C}(x-\gamma L^*y^*)\bigr\|^2+
\bigl\|y^*-J_{\gamma D^{-1}}(y^*+\gamma Lx)\bigr\|^2}{\gamma}.
\end{equation}
\end{example}

We conclude with an application of Example~\ref{ex:pd} to
Fenchel--Rockafellar duality in optimization.

\begin{example}
\label{ex:56}
Define $\YY$ and $\EEE$ as in Example~\ref{ex:pd}, let
$\varphi\in\Gamma_0(\XX)$, let $\psi\in\Gamma_0(\YY)$,
and let $L\colon\XX\to\YY$ be linear and bounded. Consider the
primal problem
\begin{equation}
\label{e:p1}
\minimize{x\in\XX}{\varphi(x)+\psi(Lx)}
\end{equation}
and the dual problem
\begin{equation}
\label{e:d1}
\minimize{y^*\in\YY^*}{\varphi^*(-L^*y^*)+\psi^*(y^*)}.
\end{equation}
Let us set $C=\partial\varphi$ and $D=\partial\psi$ in 
Example~\ref{ex:pd}, and let us define $\boldsymbol{A}$ and
$\boldsymbol{B}$ as in \eqref{e:31z}. Then
every point in the Kuhn--Tucker set
$\zer(\boldsymbol{A}+\boldsymbol{B})$ solves the primal-dual pair
\eqref{e:p1}--\eqref{e:d1}.
Now let $\gamma\in\RPP$, and let $f\in\Gamma_0(\XX)$ and 
$g\in\Gamma_0(\YY)$ be Legendre functions such that 
$\nabla f+\gamma\partial\varphi$ and 
$\nabla g^*+\gamma\partial\psi^*$ are surjective, and set
$W_{\XX}=\nabla f$ and $W_{\YY^*}=\nabla g^*$. We recall that 
$g^*$ is a Legendre function \cite[Corollary~5.5]{Ccma01}
and note that, since $f$ and $g^*$ are strictly convex on the
convex sets $\DD_{\XX}=\intdom f$ and $\DD_{\YY^*}=\intdom g^*$,
respectively, $W_{\XX}$ and $W_{\YY^*}$ are strictly monotone 
\cite[Proposition~25.10]{Zei90B}. So are therefore the operators
$W_{\XX}+\gamma C$ and $W_{\YY^*}+\gamma D^{-1}$ which, as in
Theorem~\ref{t:2}\ref{t:2ii}, makes them injective. Next, we
introduce $\Prox_{\gamma\varphi}^f=(W_{\XX}+\gamma C)^{-1}=
(\nabla f+\gamma\partial\varphi)^{-1}$ and 
$\Prox_{\gamma\psi^*}^{g^*}=(W_{\YY^*}+\gamma D^{-1})^{-1}=
(\nabla g^*+\gamma\partial\psi^*)^{-1}$. Then, as in \eqref{e:g5},
we infer from \eqref{e:fin} that, for every 
$\boldsymbol{x}=(x,y^*)\in\intdom f\times\intdom g^*$, 
\begin{align}
\label{e:99}
\vartheta(\boldsymbol{x})
&\geq\dfrac{1}{\gamma}\biggl(
D_f\brk2{x,\Prox_{\gamma\varphi}^f\brk1{\nabla f(x)-\gamma L^*y^*}}
+D_f\brk2{\Prox_{\gamma\varphi}^f
\brk1{\nabla f(x)-\gamma L^*y^*},x}\:+\nonumber\\
&\hspace{10mm}
D_{g^*}\brk2{y^*,\Prox_{\gamma\psi^*}^{g^*}
\brk1{\nabla g^*(y^*)+\gamma Lx}}
+D_{g^*}\brk2{\Prox_{\gamma\psi^*}^{g^*}
\brk1{\nabla g^*(y^*)+\gamma Lx},y^*}\biggr).
\end{align}
In particular, if $\XX$ and $\YY$ are Hilbertian, 
$f=\|\cdot\|_{\XX}^2/2$, $g=\|\cdot\|_{\YY}^2/2$, 
and $\boldsymbol{x}=(x,y^*)\in\XX\times\YY$, we obtain
\begin{equation}
\label{e:98}
\vartheta(\boldsymbol{x})
\geq\dfrac{\bigl\|x-\prox_{\gamma\varphi}(x-\gamma L^*y^*)
\bigr\|^2+\bigl\|y^*-\prox_{\gamma\psi^*}(y^*+\gamma Lx)
\bigr\|^2}{\gamma},
\end{equation}
which can also be viewed as a special case of \eqref{e:fz}. Note
that, in the Hilbertian --- and in particular the Euclidean ---
setting, $\Prox_{\gamma\varphi}^{f}$ and
$\Prox_{\gamma\psi^*}^{g^*}$ may be easier to compute than
$\prox_{\gamma\varphi}$ and $\prox_{\gamma\psi^*}$; see
\cite{Baus17,Joca16,Nguy17} and Example~\ref{ex:4}. This makes the
lower bound of \eqref{e:99} more tractable than that of
\eqref{e:98} in such instances.
\end{example}

\section{Conclusions and future directions}
\label{sec:8}

We have derived new lower bounds on the Haraux function of
set-valued operators and on the Fenchel--Young function of proper
functions in the framework of reflexive Banach spaces. These bounds
were obtained by using two types of resolvents, namely metric
resolvents and warped resolvents. Existing results have been
recovered as special cases. Two directions for future work can be
mentioned here:
\begin{itemize}
\item
In Example~\ref{ex:pd}, we have applied our results to primal-dual
composite monotone inclusions of type \eqref{e:pd}. It appears that
this approach can be used for general systems of multivariate
inclusions involving more complex monotonicity-preserving
operations such as those considered in \cite{Sadd22} and
\cite[Section~10]{Acnu24}. Indeed, as discussed in these papers,
such primal-dual systems can be reformulated in ``saddle form'' and
thus be reduced to solving a monotone inclusion of the type
$\boldsymbol{\mathsf{0}}\in\boldsymbol{\mathsf{A}}
\boldsymbol{\mathsf{x}}+\boldsymbol{\mathsf{B}}
\boldsymbol{\mathsf{x}}$ in a suitably defined product space
$\boldsymbol{\mathsf{X}}$, where 
$\boldsymbol{\mathsf{A}}\colon\boldsymbol{\mathsf{X}}\to 
2^{\boldsymbol{\mathsf{X}}^*}$ is maximally monotone and
$\boldsymbol{\mathsf{B}}\colon\boldsymbol{\mathsf{X}}\to
\boldsymbol{\mathsf{X}^*}$ is cocoercive. 
We can then exploit \eqref{e:g1} in this scenario.
\item
A question of interest is whether lower bounds on the Haraux
function can be derived in the context of nonreflexive Banach 
spaces with maximally monotone operators of type (NI) 
\cite[Section~36]{Simo08}. For instance, the lower bound
\eqref{e:14} is known to hold in this context
\cite[Theorem~2.6]{Vois09} and, in the light of
\cite[Remark~2.4]{Vois09}, it is reasonable to expect that the
results of Sections~\ref{sec:3} and \ref{sec:4} remain valid as
well.
\end{itemize}

\end{document}